\documentclass[12pt]{amsart}

\input{xy}
\xyoption{all}
\usepackage{graphicx}
\usepackage{xcolor}
\usepackage{verbatim}
\usepackage[colorlinks=true, urlcolor=blue, citecolor=red, pdfborder={0 0 0}]{hyperref}
\usepackage{amsmath}
\usepackage{amssymb}
\usepackage{tikz-cd}

\newtheorem{theorem}{Theorem} [section]
\newtheorem{prop}[theorem]{Proposition}
\newtheorem{lemma}[theorem]{Lemma}
\newtheorem{cor}[theorem]{Corollary}
\newtheorem{conjecture}[theorem]{Conjecture}

\theoremstyle{definition}
\newtheorem{remark}[theorem]{Remark}

\numberwithin{equation}{section}
\numberwithin{figure}{section}

\newcommand\A{{\mathbb A}}
\newcommand\C{{\mathbb C}}

\renewcommand\P{{\mathbb P}}

\newcommand\R{{\mathbb R}}

\newcommand\Q{{\mathbb Q}}

\newcommand\D{{\mathbb D}}
\newcommand\Gal{\operatorname{Gal}}

\newcommand\Qbar{\overline{\mathbb{Q}}}

\newcommand\Kbar{\overline{K}}
\newcommand\cM {\mathcal{M}}

\newcommand\ord{\operatorname{ord}}  %order

\newcommand\<{\langle} %innerproduct brackets 
\renewcommand\>{\rangle} %%
\newcommand\eps{\varepsilon}
\renewcommand\phi{\varphi}

\newcommand\Hyp{\mathbb{H}}
\newcommand\good{\mathrm{good}}
\newcommand\bad{\mathrm{bad}}
\newcommand\dist{\mathrm{dist}}

\definecolor{purple}{rgb}{0.5,0.0,1}

\begin{document}

\author{Laura DeMarco}
\address{
Laura DeMarco\\
Department of Mathematics\\
Northwestern University\\
2033 Sheridan Road\\
Evanston, IL 60208 \\
USA
}
\email{demarco@northwestern.edu}

\author{Holly Krieger}
\address{
Holly Krieger\\
Department of Pure Mathematics and Mathematical Statistics\\
 University of Cambridge \\
Cambridge CB3 0WB\\
UK
}
\email{hkrieger@dpmms.cam.ac.uk}

\author{Hexi Ye}
\address{
Hexi Ye\\
Department of Mathematics\\
Zhejiang University\\
Hangzhou, 310027\\
China}
\email{yehexi@gmail.com}

\title{Uniform Manin-Mumford for a family of genus 2 curves}

\begin{abstract} We introduce a general strategy for proving quantitative and uniform bounds on the number of common points of height zero for a pair of inequivalent height functions on $\mathbb{P}^1(\overline{\mathbb{Q}}).$  We apply this strategy to prove a conjecture of Bogomolov, Fu, and Tschinkel asserting uniform bounds on the number of common torsion points of elliptic curves in the case of two Legendre curves over $\mathbb{C}$.   As a consequence, we obtain two uniform bounds for a two-dimensional family of genus 2 curves: a uniform Manin-Mumford bound for the family over $\mathbb{C}$, and a uniform Bogomolov bound for the family over $\overline{\mathbb{Q}}.$
\end{abstract}

\date{\today}

\maketitle

\section{Introduction}

In this article, we use the Arakelov-Zhang intersection of adelically-metrized line bundles on $\mathbb{P}^1(\overline{\mathbb{Q}})$ to prove a uniform Manin-Mumford bound for a two-dimensional family of genus 2 curves over $\C$.  The Manin-Mumford Conjecture, proved by Raynaud \cite{Raynaud:1}, asserts the following:  Let $X$ be any smooth complex projective curve of genus $g \geq 2$, $P \in X(\C)$ any point, $j_P: X \hookrightarrow J(X)$ the Abel-Jacobi embedding of $X$ into its Jacobian $J(X)$ based at $P$, and $J(X)^{\mathit{tors}}$ the set of torsion points of the Jacobian.  Then 
\begin{equation} \label{Raynaud finiteness}
	|j_P(X) \cap J(X)^{\mathit{tors}}| < \infty.
\end{equation}
In the case of genus $g = 2,$ the curve is hyperelliptic, and the fixed points of the hyperelliptic involution provide geometrically natural choices of base point for the Abel-Jacobi map.  We show there is a uniform bound on the number of torsion images under such a map, provided the curve is also bielliptic, meaning that it admits a degree-two branched cover of an elliptic curve. 

\begin{theorem} \label{MMtheorem} There exists a uniform constant $B$ such that 
		$$|j_P(X) \cap J(X)^{\mathit{tors}}| \leq B$$
for all smooth, bielliptic curves $X$ over $\C$ of genus 2 and all Weierstrass points $P$ on $X$.  
\end{theorem}

\noindent
The curves satisfying the hypothesis of Theorem \ref{MMtheorem} form a complex surface $\mathcal{L}_2$ in the moduli space $\cM_2$ of genus 2 curves.  These $X$ are also characterized by the property that their Jacobians admit real multiplication by the real quadratic order of discriminant 4.  Further details on $\mathcal{L}_2$ are given in Section \ref{MMtheoremsection}.

\begin{remark} We do not give an explicit value for the $B$ of Theorem \ref{MMtheorem}, but this bound can be made effective by estimating the continuity constants of Section \ref{archsection}.  Poonen showed that there are infinitely many curves $X \in \mathcal{L}_2$ for which $|j_P(X) \cap J(X)^{\mathit{tors}}|$ is at least 22, taking $P$ to be a Weierstrass point on $X$  \cite[Theorem 1]{Poonen:1}.  More recently, Stoll found an example with $|j_P(X) \cap J(X)^{\mathit{tors}}| = 34$ for Weierstrass point $P$ \cite{Stoll:example}; the curve $X$ is defined over $\Q$.  We know of no curve $X \in \mathcal{M}_2(\C)$ and point $P \in X$ satisfying $|j_P(X) \cap J(X)^{\mathit{tors}}|>34$.
\end{remark}

The question of uniformity in \eqref{Raynaud finiteness} was raised by Mazur in \cite{Mazur:curves} who asked if a bound could be given that depends only on the genus $g$ of the curve $X$.  Quantitative bounds on torsion points on curves have been obtained when the curve is defined over a number field, notably by Coleman \cite{Coleman}, Buium \cite{Buium}, Hrushovski \cite{Hrushovski}, and more recently by Katz, Rabinoff, and Zureick-Brown \cite{KRZB}.   By quantifying the $p$-adic approach to \eqref{Raynaud finiteness}, these authors achieve bounds for general families of curves; however, these bounds all involve dependence on field of definition or the choice of a prime for the family of curves, so are not uniform for families over $\Qbar$ or $\mathbb{C}$.  

Our new technique which yields Theorem \ref{MMtheorem} is a quantification of the approach of Szpiro, Ullmo, and Zhang \cite{Szpiro:Ullmo:Zhang, Ullmo:Bogomolov, Zhang:Bogomolov} to proving \eqref{Raynaud finiteness}, utilizing adelic equidistribution theory.  We first reduce to the setting where the curve is defined over $\Qbar$.  Over $\Qbar$, we build on the proof of the quantitative equidistribution theorem for height functions on $\P^1(\Qbar)$ of Favre and Rivera-Letelier \cite{FRL:equidistribution}.  

In fact, we deduce Theorem \ref{MMtheorem} from a case of the following conjecture, discussed by Bogomolov and Tschinkel \cite{Bogomolov:Tschinkel:small} and stated formally as \cite[Conjectures 2 and 12]{BFT:torsion}, which asserts uniform bounds on common torsion points for pairs of elliptic curves.  By a {\em standard projection} $\pi: E\to \P^1$ of an elliptic curve $E$ over $\C$, we mean any degree-two quotient that identifies a point $P$ and its inverse $-P$.  Note that $\pi$ has a simple critical point at each of the four elements of the 2-torsion subgroup $E[2]$.

\begin{conjecture} \cite{BFT:torsion} \label{BFTconj}  There exists a uniform constant $B$ such that 
$$|  \pi_1(E_1^{\mathit{tors}}) \cap \pi_2(E_2^{\mathit{tors}}) | \leq B$$
for any pair of elliptic curves $E_i$ over $\C$ and any pair of standard projections $\pi_i$ for which 
$\pi_1(E_1[2]) \not= \pi_2(E_2[2])$.
\end{conjecture}

\noindent
Note that if $\pi_1(E_1[2]) = \pi_2(E_2[2])$, then $E_1$ is isomorphic to $E_2$ and $\pi_1(E_1^{\mathit{tors}}) = \pi_2(E_2^{\mathit{tors}})$.  The finiteness of the set $\pi_1(E_1^{\mathit{tors}})  \cap  \pi_2(E_2^{\mathit{tors}})$, under the assumption that $\pi_1(E_1[2]) \not= \pi_2(E_2[2])$, follows from the main theorem of Raynaud in \cite{Raynaud:1}; indeed, the diagonal in $\P^1\times\P^1$ lifts to a (singular) curve $C \subset E_1\times E_2$ via $\pi_1 \times \pi_2$ with normalization of genus $g \geq 2$ \cite{Bogomolov:Tschinkel:small}.  

We prove Conjecture \ref{BFTconj} in the case of maximal overlap of the 2-torsion points; i.e., when
	$$|\pi_1(E_1[2]) \cap \pi_2(E_2[2])| = 3.$$
This setting corresponds to the case where the (normalization of the) curve $C$ in $E_1 \times E_2$ has  genus 2.  By fixing coordinates on $\P^1$, it suffices to work with the Legendre family of elliptic curves 
\begin{equation} \label{Leg}
	E_t \; : \; y^2 = x(x-1)(x-t)
\end{equation}
with $t \in \mathbb{C} \setminus \{ 0,1\}$ and the standard projection $\pi(x,y) = x$ on $E_t$.  (See Corollary \ref{BFT special case}.)

\begin{theorem} \label{Legendretheorem} There exists a uniform constant $B$ such that
$$|\pi(E_{t_1}^{\mathit{tors}}) \cap \pi(E_{t_2}^{\mathit{tors}})| \leq B,$$
for all $t_1 \not= t_2$ in $\C\setminus\{0,1\}$, for the curves $E_t$ defined by \eqref{Leg} and projection $\pi(x,y) = x$.
\end{theorem}

To prove Theorem \ref{Legendretheorem}, we introduce a general strategy for bounding the number of common height-zero points for any pair of distinct height functions $h_1, h_2: \P^1(\Qbar) \to \R$ that arise from continuous, semipositive, adelic metrics on the line bundle $\mathcal{O}_{\P^1}(1)$.  There is a natural Arakelov-Zhang pairing between any two such heights, given by the intersection number of the associated metrized line bundles.  Our heights are normalized so this intersection number, which we denote by $h_1 \cdot h_2$, will satisfy 
	$$h_1 \cdot h_2 \geq 0 \mbox{ with equality if and only if } h_1 = h_2.$$  
Details on these heights and the pairing are given in Section \ref{background}.  The value of $h_1 \cdot h_2$ provides a notion of distance between the two heights (as was observed by Fili in \cite{Fili:energy}).   It follows from equidistribution \cite{ChambertLoir:equidistribution, FRL:equidistribution, BRbook} that
\begin{equation} \label{equilimit}	\lim_{n\to\infty} h_2(x_n) =   h_1 \cdot h_2
\end{equation}
for any infinite sequence of distinct points $x_n \in \P^1(\Qbar)$ such that $h_1(x_n) \to 0$ as $n\to\infty$, suggesting that large numbers of common zeroes between $h_1$ and $h_2$ will imply that $h_1$ and $h_2$ are close.  However, this measure of closeness between two heights is not generally uniform in families of heights, because the rate of equidistribution is not uniform.  Nevertheless, by bounding the height pairing $h_1 \cdot h_2$ from below, we can obtain an upper bound on the number of common zeroes for certain families.

In the context of Theorem \ref{Legendretheorem}, we consider the family of height functions $\hat{h}_t$ on $\P^1(\Qbar)$ induced from the N\'eron-Tate canonical height on the elliptic curve $E_t$, for $t\in \Qbar\setminus\{0,1\}$; its zeroes are precisely the elements of $\pi(E_t^{\mathit{tors}})$.  We implement this general strategy by proving three bounds on the intersection pairing $\hat{h}_{t_1} \cdot \hat{h}_{t_2}$.  We prove a uniform lower bound on the pairing:

\begin{theorem} \label{uniformlowerbound} There exists $\delta > 0$ such that
$$\hat{h}_{t_1} \cdot \hat{h}_{t_2} \geq \delta$$
for all $t_1 \ne t_2 \in \overline{\mathbb{Q}} \setminus \{ 0, 1\}$.  
\end{theorem}

\noindent
We also prove an asymptotic lower bound for parameters $t_1$ and $t_2$ with large height:

\begin{theorem} \label{pairingasymptotics} There exist constants $\alpha, \beta > 0$ such that 
$$\hat{h}_{t_1} \cdot \hat{h}_{t_2} \; \geq \; \alpha \, h(t_1, t_2) - \beta$$
for all $t_1 \not= t_2$ in $\overline{\mathbb{Q}} \setminus \{0,1\}$.  Here $h(t_1, t_2)$ is the naive logarithmic height on $\mathbb{A}^2(\Qbar)$.  
\end{theorem}

\noindent
We find an upper bound that depends on the number of common zeroes of $\hat{h}_{t_1}$ and $\hat{h}_{t_2}$ as well as the heights of the parameters $t_1$ and $t_2$:

\begin{theorem} \label{boundingN} For all $\eps > 0$, there exists a constant $C(\eps) > 0$ such that
$$\hat{h}_{t_1} \cdot \hat{h}_{t_2} \leq \left( \eps + \frac{C(\eps)}{N(t_1, t_2)} \right) (h(t_1, t_2) + 1)$$
for all $t_1$ and $t_2$ in $\overline{\mathbb{Q}} \setminus \{0, 1\}$, where $N(t_1, t_2) := |\pi(E_{t_1}^{\mathit{tors}}) \cap \pi(E_{t_2}^{\mathit{tors}})|$.
\end{theorem}

\noindent
The three theorems combine to give a uniform bound on the number $N(t_1, t_2)$ of common zeroes of $\hat{h}_{t_1}$ and $\hat{h}_{t_2}$ for all $t_1 \not= t_2$ in $\Qbar\setminus\{0,1\}$.

Theorems \ref{uniformlowerbound} and \ref{pairingasymptotics} follow from estimates on the local height functions and the local equilibrium measures on the $v$-adic Berkovich projective line at each place $v$ of a number field containing $t_1$ and $t_2$, computed using the dynamical Latt\`es map $f_t: \mathbb{P}^1 \rightarrow \mathbb{P}^1$ induced by multiplication by $2$ on a Legendre curve $E_t$.  The non-archimedean contributions to $\hat{h}_{t_1}\cdot \hat{h}_{t_2}$ turn out to be straightforward to compute for these heights.  Significant technical issues arise when $v$ is archimedean and both parameters $t_i$ are tending to the singularity set $\{0,1,\infty\}$ for this family; we resolve these issues by appealing to the theory of degenerations of complex dynamical systems on $\P^1(\C)$, in which a family of complex rational maps degenerates to a non-archimedean dynamical system acting on a Berkovich space, as in the work of DeMarco-Faber \cite{DF:degenerations} and Favre \cite{Favre:degenerations}, using the formalism of hybrid space as discussed by Boucksom-Jonsson in \cite{Boucksom:Jonsson:hybrid}.  

For Theorem \ref{boundingN}, we expand upon the quantitative equidistribution results of Favre-Rivera-Letelier \cite{FRL:equidistribution} and Fili \cite{Fili:energy} to analyze the rates of convergence of measures supported on finite sets of zeroes of a height $h$ to the associated equilibrium measures at each place $v$.  To do so requires control on the modulus of continuity of the local heights, and again we rely on estimates from the hybrid space to treat the cases where a parameter $t$ is tending to one of the singularities for the family $E_t$.

Although Theorem \ref{uniformlowerbound} alone was not enough to prove Theorem \ref{Legendretheorem}, it implies a uniform bound of a different sort, when combined with Zhang's inequality on the essential minimum of a height function \cite{Zhang:adelic}:

\begin{prop}  \label{uniformZhang}
Choose any $b$ satisfying $0 < b < \delta/2$ for the $\delta$ of Theorem \ref{uniformlowerbound}.  
Then the set 
	$$S(b, t_1, t_2) := \{ x \in \P^1(\Qbar):  \hat{h}_{t_1}(x) + \hat{h}_{t_2}(x) \leq b\}$$
is finite for each pair $t_1 \not= t_2 \in \Qbar\setminus\{0,1\}$.
\end{prop}

The complete proof of Theorem \ref{Legendretheorem}, however, gives more:  we obtain a uniform bound on the size of the set $S(b,t_1, t_2)$ defined in Proposition \ref{uniformZhang} (see Theorem \ref{LegendreBogomolov}). This in turn provides a uniform version of the Bogomolov Conjecture for the associated family of genus 2 curves.  The Bogomolov Conjecture was proved for each individual curve $X$ over $\Qbar$ in \cite{Ullmo:Bogomolov, Zhang:Bogomolov}.  To state our result precisely, we fix ample and symmetric line bundles on the family of Jacobians $J(X)$ for the genus 2 curves $X$ defined over $\Qbar$ that we consider in Theorem \ref{MMtheorem}.  Specifically, we take $L_X = \Phi^* L_D$ for the isogeny $\Phi: J(X) \to E_1\times E_2$ of Proposition \ref{surface}, with $L_D$ the line bundle associated to the divisor $D = \{O_1\}\times E_2 + E_1 \times \{O_2\}$, where $O_i$ is the identity element of $E_i$.  

\begin{theorem} \label{uniformBogomolov}
There exist constants $B$ and $b>0$ such that 
		$$\big| \{x \in j_P(X)(\Qbar) : \hat{h}_{L_X}(x) \leq b\} \big| \; \leq \; B$$
for all smooth curves $X$ over $\Qbar$ of genus 2 admitting a degree-two map to an elliptic curve and all Weierstrass points $P$ on $X$, where $\hat{h}_{L_X}$ is the N\'eron-Tate canonical height on the Jacobian $J(X)$.
\end{theorem}

Finally, we mention that we implement this general strategy towards uniform boundedness in a follow-up article \cite{DKY:quadpoly} in another setting, providing a uniform bound on the number of common preperiodic points for distinct polynomials of the form $f_c(z) = z^2+c$ with $c\in \C$.

\begin{remark}
We have chosen to work with the Arakelov-Zhang pairing $\hat{h}_{t_1} \cdot \hat{h}_{t_2}$ to measure proximity of the two height functions, with $t_1 \not= t_2$ in $\Qbar\setminus\{0,1\}$, but there are other choices we could have made.  For example, Kawaguchi and Silverman in \cite{Kawaguchi:Silverman:pairing} study
	$$\delta\left(\hat{h}_{t_1},\hat{h}_{t_2}\right) := \sup_{x \in \P^1(\Qbar)} \left|\hat{h}_{t_1}(x) - \hat{h}_{t_2}(x) \right|.$$
It turns out that the two quantities are comparable for this family of heights.  The upper bound $\hat{h}_{t_1} \cdot \hat{h}_{t_2} \leq \; \delta\left(\hat{h}_{t_1},\hat{h}_{t_2}\right)$ can be seen as a corollary of arithmetic equidistribution and \eqref{equilimit}, and therefore holds for any pair of normalized heights coming from continuous, semipositive adelic metrics on $\mathcal{O}_{\P^1}(1)$.  A lower bound of the form $\hat{h}_{t_1} \cdot \hat{h}_{t_2} \geq C_1 \; \delta\left(\hat{h}_{t_1},\hat{h}_{t_2}\right) - C_2$ for real constants $C_1, C_2$, and for all $t_1 \not= t_2$ in $\Qbar\setminus\{0,1\}$, is a consequence of Theorem \ref{pairingasymptotics}, when combined with \cite[Theorem 1]{Kawaguchi:Silverman:pairing}.  However, such a lower bound does not hold for all pairs of heights coming from metrics on $\mathcal{O}_{\P^1}(1)$.  A comparison of these two pairings is addressed further in \cite{DKY:quadpoly}, for the canonical heights associated to morphisms of $\P^1$.
\end{remark}

\medskip\noindent
{\bf Outline of the paper.} We fix our notation and provide background in Section \ref{background}.   Sections \ref{nonarchsection},  \ref{archsection}, and \ref{arch energies} provide the estimates on local height functions and local measures needed to prove all of our theorems.  Theorem \ref{pairingasymptotics} is proved in Section \ref{pairingsection}, and from it we deduce Theorem \ref{uniformlowerbound} and Proposition \ref{uniformZhang}.  A generalization of Theorem \ref{boundingN} is proved in Section \ref{trianglesection} which treats points of small height, not only of height 0.  We prove Theorem \ref{Legendretheorem} in Section \ref{Legendresection} and finally Theorems \ref{MMtheorem} and \ref{uniformBogomolov} in Section \ref{MMtheoremsection}.  

\medskip\noindent
{\bf Acknowledgements.}  We thank the American Institute of Mathematics, where the initial work for this paper took place as part of an AIM SQuaRE.  Special thanks go to Ken Jacobs, Mattias Jonsson, Curt McMullen, and Khoa Nguyen for helpful discussions.  We would also like to thank the anonymous referees for their many suggestions and careful reading of the article.  During the preparation of this paper, L. DeMarco was supported by the National Science Foundation (DMS-1600718), H. Krieger was supported by Isaac Newton Trust (RG74916), and H. Ye was partially supported by ZJNSF (LR18A010001) and NSFC (11701508).

%%%%%%
\bigskip
\section{Heights, measures, and energies}
\label{background}

This section develops the background and notation needed for the proofs that follow.  Throughout, $K$ is a number field and $M_K$  its set of places.

\subsection{The canonical height} \label{canonical height}
Fix $t\in \Qbar\setminus\{0,1\}$.  Let $E_t$ be the Legendre elliptic curve and $\pi: E_t\to \P^1$ the projection defined by $\pi(x,y) = x$.  The multiplication-by-two endomorphism on $E_t$ descends via $\pi$ to a morphism of degree 4 on $\P^1$ given by
\begin{equation} \label{Lattes map}
	f_t(x)=\frac{(x^2-t)^2}{4x(x-1)(x-t)}.
\end{equation} 
The canonical height on the elliptic curve 
	$$\hat{h}_{E_t} : E_t(\Qbar) \to \R$$
can be defined via the projection $\pi$ and the iteration of $f_t$ as $h_{E_t}(P):= \frac12 \hat{h}_t(\pi(P))$
where
	$$\hat{h}_t : \P^1(\Qbar) \to \R$$ 
is the dynamical canonical height defined by 
\begin{equation} \label{dynamical canonical height}
	\hat{h}_t(x) = \lim_{n\to\infty}  \frac{1}{4^n}  h(f_t^n(x)).
\end{equation}
Here, $h$ is the (logarithmic) Weil height on $\P^1(\Qbar)$.  Note that $\hat{h}_t(x) \geq 0$ for all $x \in \P^1(\Qbar)$, and
	$$\hat{h}_t(x) = 0 \iff x \in \pi(E_t^{\mathit{tors}})$$
\cite{Silverman:elliptic}, \cite{Call:Silverman}.

The height $\hat{h}_t$ has a local decomposition as follows:  for any number field $K$ containing $t$, and for each place $v \in M_K$, there exists a local height function $\lambda_{t,v}$ such that
$$\hat{h}_t(x) =  \sum_{v \in M_K} \frac{r_v}{|\Gal(\Kbar/K) \cdot x|} \sum_{y \in \Gal(\Kbar/K) \cdot x} \lambda_{t,v}(y) $$
for all $x \in \Kbar$, where 
	$$r_v := \frac{[K_v: \Q_v]}{[K:\Q]}.$$
The local heights $\lambda_{t,v}$ can be chosen to extend continuously to $\P^1(\C_v)\setminus\{\infty\}$, where $\C_v$ is the completion (w.r.t.~$v$) of an algebraic closure of the completion $K_v$, and to satisfy 
$$\lambda_{t,v}(x) = \log|x|_v + \mathcal{O}(1)$$ 
as $|x|_v \to \infty$.  

\subsection{Local heights and escape rates}
To compute the local heights, we will often express the maps $f_t: \P^1\to \P^1$ of \eqref{Lattes map} in homogeneous coordinates, as 
  $$F_t(z,w):=\left( (z^2-tw^2)^2, 4 zw(z-w)(z-tw) \right)$$
for $z$ and $w$ in $\C_v$.  As observed in \cite[Chapter 10]{BRbook}, its escape-rate function
\begin{equation} \label{escape rate}
  G_{F_t,v}(z,w):=\lim_{n\to \infty}\frac{1}{4^n} \log\|F_t^n(z,w)\|_v,
\end{equation}
where $\|(z,w) \|_v= \max\{|x|_v, |y|_v\}$, satisfies 
$$\hat{h}_t(x) = \sum_{v \in M_K} \frac{r_v}{|\Gal(\Kbar/K) \cdot \tilde{x}|} \sum_{\tilde{y} \in \Gal(\Kbar/K) \cdot \tilde{x}} G_{F_t,v}(\tilde{y})$$
for $x \in \P^1(\Kbar)$ and $\tilde{x}$ any choice of lift of $x$ to $\Kbar^2\setminus\{(0,0)\}$.  In particular, we may take 
\begin{equation} \label{local height definition}
	\lambda_{t,v}(x) = G_{F_t, v}(x,1)
\end{equation}
as a local height for $\hat{h}_t$.

The elliptic curves $E_t$ and $E_{1-t}$ and $E_{1/t}$ are isomorphic, with the following transformation formulas for the local heights:

\begin{prop} \label{G transformation}
Fix any number field $K$ and $v \in M_K$.  Then, for all $t\in K\setminus\{0,1\}$, we have 
	$$G_{F_{1-t}, v}(1-z, 1) = G_{F_t, v}(z,1) = G_{F_{1/t}, v}(z, t) = G_{F_{1/t}, v}(z/t, 1) + \log|t|_v.$$
\end{prop}

\begin{proof}
Let $A$ be the automorphism $A(z,w) = (w-z, w)$.  Then 
	$$A\circ F_t^n  = - F_{1-t}^n \circ A$$
for all iterates, proving the first equality.  Similarly, let $B(z,w) = (z, t w)$.  Then
	$$B \circ F_t^n  = F_{1/t}^n \circ B$$
for all iterates, proving the second equality.  The final equality follows from the logarithmic homogeneity of $G$.
\end{proof}

\subsection{The Berkovich projective line} 
Let $K$ be a number field.  For each $v \in M_K$, let $\A^{1,an}_v$ denote the Berkovich affine line over $\C_v$.  For non-archimedean $v$, the points of $\A^{1,an}_v$ come in four types.  The Type I points in $\A_v^{1,an}$ are, by definition, the elements of the field $\C_v$.  The Type II points are in one-to-one correspondence with disks $D(a,r) = \{x \in \C_v:  |x-a|_v \leq r\}$ with $r >0$ rational, and these are the branch points for the underlying tree structure on $\A^{1,an}_v$.  The Type III points correspond to disks $D(a,r)$ with $r$ irrational.  (We will not need the Type IV points in this article.)  A Type II or III point corresponding to $D(a,r)$ will be denoted by $\zeta_{a, r}$.  The {\em Gauss point} $\zeta_{0,1}$ is the Type II point identified with the unit disk.  The Berkovich projective line $\P_v^{1,an}=\A_v^{1,an}\cup \{\infty\}$ is the one-point compactification of $\A_v^{1,an}$, which is a canonically-defined path-connected compact Hausdorff space containing $\P^1(\C_v)$ as a dense subspace.  If $v$ is archimedean, then $\C_v\simeq \C$ and $\P^{1,an}_v=\P^1(\C)$.

For each $v \in M_K$ there is a distribution-valued Laplacian operator $\Delta$ on $\P_v^{1,an}$. The function $\log ^+|z|_v$ on $\P^1(\C_v)$ extends naturally to a continuous real valued function $\P_v^{1,an}\to \R \cup \{\infty\}$, and the Laplacian is normalized such that
  $$\Delta \log ^+|z|_v=\omega_v - \delta_\infty$$ 
on $\P^{1,an}_v$, where $\omega_v = m_{S^1}$ is the Lebesgue probability measure on the unit circle when $v$ is archimedean, and $\omega_v = \delta_G$ is a point mass at the Gauss point of $\P^{1,an}_v$ when $v$ is non-archimedean.  A probability measure $\mu_v$ on $\P^{1,an}_v$ is said to have {\em continuous potentials} if $\mu_v-\omega_v=\Delta g$ with $g: \P^{1,an}_v\to \R$ continuous.  The function $g$ for $\mu_v$ is unique up to the addition of a constant.  See \cite[Chapter 5]{BRbook} for more details.  Note that the Laplacian used here is the negative of the one appearing in \cite{PST:pairing} and \cite{BRbook}, but agrees with the usual Laplacian (up to a factor of $2\pi$) at the archimedean places.

For $v$ non-archimedean, we set
   $$\Hyp:=\A_v^{1,an}\setminus \C_v.$$
The hyperbolic distance $d_{\mathrm{hyp}}$ on $\Hyp$ gives it the structure of a metrized $\R$-tree and satisfies 
   $$d_{\mathrm{hyp}}(\zeta_{a,r_1}, \zeta_{a,r_2}) =  \log (r_1/r_2) $$
for any $a\in \C_v$ and any $r_1 \geq r_2 > 0$.  We will say that a probability measure $\mu_v$ on $\Hyp$ is an {\em interval measure} if it is the uniform distribution on an interval $[\zeta_1, \zeta_2]\subset \Hyp$ with respect to the linear structure induced from the hyperbolic metric $d_{\mathrm{hyp}}$.

\subsection{Canonical measures and good reduction} \label{canonical measures}
For each Legendre elliptic curve $E_t$ with $t$ in a number field $K$ and each $v \in M_K$, the local height $\lambda_{t,v}$ of \eqref{local height definition} extends to define a continuous and subharmonic function on $\A_v^{1,an}$ with logarithmic singularity at $\infty$.  We have
   $$\Delta \lambda_{t, v}=\mu_{t,v} - \delta_\infty$$
on $\P_v^{1,an}$, where $\mu_{t, v}$ is the canonical probability measure for the dynamical system $f_t$ at $v$ \cite{FRL:equidistribution}, \cite[Theorem 10.2]{BRbook}.

For archimedean $v \in M_K$, the measure $\mu_{t,v}$ is the unique $f_t$-invariant measure on $\P^1(\C)$ achieving the maximal entropy $\log 4$.  It is the push-forward of the Haar measure on $E_t(\C)$ via the projection $\pi$ introduced in \S\ref{canonical height}.  See, for example, \cite{Milnor:Lattes} for a dynamical discussion of the maps $f_t$ on the Riemann sphere.

For non-archimedean $v \in M_K$, if the curve $E_t$ and the map $f_t$ have good reduction, the measure $\mu_{t,v}$ is the point mass $\delta_G$ supported on the Gauss point $\zeta_{0,1}$.  The map $f_t$ has potential good reduction, meaning that it has good reduction under a suitable change of coordinates on $\P^1$, if and only if the measure $\mu_{t,v}$ is supported at a single Type II point in $\Hyp$.  In general, the support of $\mu_{t,v}$ is equal to the Julia set of $f_t$ in $\P^{1,an}_v$.  

Recall that the $j$-invariant of the elliptic curve $E_t$ over $\mathbb{C}$ is given by
\begin{equation} \label{j invariant}
   j(t)=\frac{256(1-t+t^2)^3}{(1-t)^2t^2}.
\end{equation}
For $t \in K$ and non-archimedean $v \in M_K$, the map $f_t$ has potential good reduction at $v$ if and only if the curve $E_t$ has potential good reduction at $v$.  This equivalence can be proved via equidistribution of torsion points on $E_t$ at all places \cite[Theorem 1]{Baker:Petsche} (thus implying that the measure $\mu_{t,v}$ will also be supported at a single point of $\P^{1,an}_v$) or via a direct calculation showing that the Julia set of $f_t$ is a singleton if and only if $|j(t)|_v \leq 1$.

\subsection{The height as an adelic metric}
Suppose $t \in K \setminus \{0, 1\}$.  The height $\hat{h}_t$ on $\P^1(\Qbar)$, introduced in \S\ref{canonical height}, is induced from an adelic metric on $\mathcal{O}_{\P^1}(1)$, in the sense of Zhang \cite{Zhang:adelic}.   Fixing coordinates on $\P^1$ and a section $s$ of $\mathcal{O}_{\P^1}(1)$ with $(s) = (\infty)$, then a metric $\|\cdot\|_{t,v}$ can be defined at each place $v$ of $K$ by setting 
	$$-\log \|s(z)\|_{t,v} =  \lambda_{t,v}(z) = G_{F_t,v}(z,1),$$
for the function $G_{F_t,v}$ of \eqref{escape rate}.  The height $\hat{h}_t$ satisfies
	$$\hat{h}_t(x) = \sum_{v \in M_K} \frac{r_v}{|\Gal(\Kbar/K) \cdot x|} \sum_{y \in \Gal(\Kbar/K) \cdot x} (-\log \|s(y)\|_{t,v})$$
for all $x \not= \infty$ in $\P^1(\Qbar)$.  Writing $\lambda_{t,v}(z) = \log|z|_v + c_v + o(1)$ as $|z|_v \to \infty$ with a constant $c_v$ at each place $v$ of $K$, we may compute that
\begin{equation} \label{height at infinity}
	0 = \hat{h}_t(\infty) = \sum_{v \in M_K} r_v \, c_v,
\end{equation}
because $\infty$ is the projection of the origin of $E_t$.  

\subsection{The intersection pairing}
For these heights $\hat{h}_t$ coming from the Legendre family of elliptic curves, with $t \in \Qbar\setminus\{0,1\}$, we have 
\begin{equation} \label{unique heights}
	\hat{h}_{t_1} = \hat{h}_{t_2} \iff t_1 = t_2.
\end{equation}
Indeed, any height coming from an adelic metric on $\mathcal{O}_{\P^1}(1)$ is uniquely determined, up to an additive constant, by the associated curvature distributions; see, for example, the construction of a height function from the measures in \cite{FRL:equidistribution}.  For heights of the form $\hat{h}_t$, at each archimedean place $v$ of a number field containing $t$, the curvature distribution $\mu_{t,v}$ on $\P^1(\C)$ is the push-forward of the Haar measure on $E_t(\C)$ by $\pi$; it therefore has a greater density at the four branch points $\{0,1,t, \infty\}$ of $\pi$, and thus determines $t$.  

There is a well-defined intersection number between any pair of such heights, as in \cite{Zhang:adelic} (see also \cite{ChambertLoir:survey}); more precisely, it is the arithmetic intersection number of the two associated adelically metrized line bundles.  By the non-degeneracy of this height pairing and \eqref{unique heights},
\begin{equation} \label{non-degeneracy}
	\hat{h}_{t_1} \cdot \hat{h}_{t_2} \geq 0 \mbox{ with equality if and only if } t_1 = t_2,
\end{equation}
as computed in \cite{PST:pairing}.  

To define the pairing $\hat{h}_{t_1} \cdot \hat{h}_{t_2}$, we fix sections $s$ and $u$ of $\mathcal{O}_{\P^1}(1)$ such that their divisors do not intersect.  Given $t_1$ and $t_2$ in a number field $K$, and a place $v$ of $K$, we set
	$$\<\hat{h}_{t_1}, \hat{h}_{t_2}\>^{s,u}_v :=  \int \log\|s\|^{-1}_{t_1,v} \; \Delta(\log\|u\|^{-1}_{t_2,v}) = \< \hat{h}_{t_2}, \hat{h}_{t_1} \>^{u,s}_v.$$
The integral is over the Berkovich analytification $\P^{1,an}_v$ of $\P^1$, over the field $\C_v$.  The metrics satisfy 
	$$\Delta(\log\|s\|^{-1}_{t,v}) = \mu_{t,v} - \delta_{(s)},$$
and $\mu_{t,v}$ is the associated curvature distribution.  

The height pairing is then defined as 
\begin{equation} \label{height pairing definition}
\hat{h}_{t_1} \cdot \hat{h}_{t_2} := \hat{h}_{t_1}((u)) + \hat{h}_{t_2}((s)) +  \sum_{v \in M_K} r_v \; \<\hat{h}_{t_1}, \hat{h}_{t_2}\>^{s,u}_v
\end{equation}
which is independent of the choices of $s$ and $u$.  This pairing is easily seen to be symmetric, and since $\hat{h}_t(\infty) = 0$ for all $t$, it can be expressed as
\begin{eqnarray} \label{pairing with infinity}
\hat{h}_{t_1} \cdot \hat{h}_{t_2} &=& \hat{h}_{t_2}(\infty) + \sum_{v \in M_K} r_v \int (\log\|s\|^{-1}_{t_1,v}) \, d\mu_{t_2,v} \; = \; \sum_{v \in M_K} r_v \int \lambda_{t_1,v} \, d\mu_{t_2,v}  \nonumber \\
	&=&  \hat{h}_{t_1}(\infty) + \sum_{v \in M_K} r_v \int (\log\|s\|^{-1}_{t_2,v}) \, d\mu_{t_1,v} \; = \; \sum_{v \in M_K} r_v \int \lambda_{t_2,v} \, d\mu_{t_1,v}
\end{eqnarray}
when $(s) = (\infty)$.  

As $\hat{h}_t \cdot \hat{h}_t = 0$ for all $t \in \Qbar\setminus\{0,1\}$ from \eqref{non-degeneracy}, note that
\begin{equation}  \label{zero sum}
	\sum_{v \in M_K} r_v \int \lambda_{t,v} \, d\mu_{t,v} = 0 =   \sum_{v \in M_K} r_v \, c_v,
\end{equation}
by combining \eqref{pairing with infinity} and \eqref{height at infinity}.  The pairing can be rewritten as:
\begin{eqnarray} \label{pairing definition}
\hat{h}_{t_1} \cdot \hat{h}_{t_2}  &=&  
  \frac12 \left( \hat{h}_{t_2}(\infty) +  \hat{h}_{t_1}(\infty)  +  \sum_{v \in M_K} r_v  \left( \int \lambda_{t_1, v} \, d\mu_{t_2,v} + \int \lambda_{t_2, v} \, d\mu_{t_1,v} \right) \right)  \nonumber \\
  &=& 
   \frac12 \sum_{v \in M_K} r_v  \left(  \int \left(\lambda_{t_1, v} - \lambda_{t_2,v}\right) \, d\mu_{t_2,v} + \int \left( \lambda_{t_2, v} - \lambda_{t_1, v} \right) \, d\mu_{t_1,v} \right). 
\end{eqnarray}
The advantage of working with \eqref{pairing definition} is the following local version of the non-degeneracy property \eqref{non-degeneracy}:  

\begin{prop}  \cite[Propositions 2.6 and 4.5]{FRL:equidistribution} \label{local energy}
Let $K$ be a number field and $v \in M_K$.  For any $t_1, t_2 \in K \setminus\{0,1\}$, the local energy
	$$E_v(t_1, t_2) := \frac{1}{2}\left( \int \left(\lambda_{t_1, v} - \lambda_{t_2,v}\right) \, d\mu_{t_2,v} + \int \left( \lambda_{t_2, v} - \lambda_{t_1, v} \right) \, d\mu_{t_1,v}\right)$$
is non-negative; it is equal to $0$ if and only if $\mu_{t_1,v} = \mu_{t_2,v}$.  
\end{prop}

\begin{prop}\label{one sing}
Let $v \in M_K$, and fix $t_1, t_2 \in K\setminus\{0,1\}$. We have 
  $$E_v(t_2, t_1)=E_v(t_1, t_2)=E_v(1-t_1, 1-t_2)=E_v(1/t_1, 1/t_2).$$
\end{prop}

\begin{proof}
Given measures $\mu_{t_1, v}$ and $\mu_{t_2, v}$, the local energy $E_v(t_1, t_2)$ can be expressed as 
	$$- \frac12 \int g \;d(\mu_{t_1,v} - \mu_{t_2, v})$$
for any continuous potential $g$ of the signed measure $\mu_{t_1,v} - \mu_{t_2, v}$, because $g = \lambda_{t_1, v} - \lambda_{t_2,v} + c$ for some constant $c$.  We have 
   $$f_{1-t} =\alpha \circ  f_t\circ \alpha^{-1}$$
for $\alpha(z)=1-z = \alpha^{-1}(z)$, such that $\mu_{1-t,v} = \alpha_* \mu_{t,v}$ and $g = (\lambda_{t_1, v} - \lambda_{t_2,v}) \circ \alpha^{-1}$ is a potential for the measure $\mu_{1-t_1, v} - \mu_{1 - t_2, v}$.  Therefore, $E_v(1 - t_1, 1-t_2) = E_v(t_1, t_2)$.  Similarly, we have $f_{1/t}(z)=\alpha\circ  f_t\circ \alpha^{-1} (z)$ for $\alpha(z)=z/t$, so $E_v(1/t_1, 1/t_2) = E_v(t_1, t_2)$.  
\end{proof}

\subsection{Measures and mutual energy}\label{Mutualenergy}
Suppose that $\nu_1$ and $\nu_2$ are signed measures on $\P^1(\C)$ with trace measures $|\nu_i|$ for which the function $\log|z-w| \in L^1(|\nu_1|\otimes |\nu_2|)$ on $\C^2\setminus\mathrm{Diag}$.  The mutual energy of $\nu_1$ and $\nu_2$ is defined in \cite{FRL:equidistribution} by 
\begin{equation} \label{archimedean pairing}
	(\nu_1, \nu_2) := - \int_{\C^2\setminus\mathrm{Diag}} \log|z-w| \, d\nu_1 \otimes d\nu_2.
\end{equation}
This definition extends to the non-archimedean setting by replacing $|z-w|$ with the Hsia kernel $\delta_v(z,w)$ based at the point at $\infty$.  In this way, for $v \in M_K^0$, a pairing is defined similarly as
\begin{equation} \label{non-archimedean pairing}
	(\nu_1, \nu_2)_v := - \int_{\A^{1,an}_v \times \A^{1,an}_v \setminus\mathrm{Diag}} \log \delta_v(z,w) \, d\nu_1 \otimes d\nu_2.
\end{equation}
See \cite[\S4.4]{FRL:equidistribution} and \cite[Chapter 4]{BRbook}.  

For measures $\nu_i$ of total mass 0 with continuous potentials on $\P^{1,an}_v$ we have
	$$(\nu_1, \nu_2)_v = - \int g_1 \, d\nu_2$$
for any choice of continuous potential $g_1$ for $\nu_1$.  Further, $(\nu_1, \nu_2)_v \geq 0$ with equality if and only if $\nu_1 = \nu_2$ \cite[Propositions 2.6 and 4.5]{FRL:equidistribution}.  Note that Proposition \ref{local energy} is a special case of this fact.  Indeed, in this notation, the local energy $E_v(t_1, t_2)$ defined in Proposition \ref{local energy} is given by 
\begin{equation} \label{local energy as measure pair}
	E_v(t_1, t_2) = \frac{1}{2} \, (\mu_{t_1,v} - \mu_{t_2, v}, \; \mu_{t_1,v} - \mu_{t_2, v})_v
\end{equation}
at each place $v$ of a number field containing $t_1$ and $t_2$, for the canonical measures introduced in \S\ref{canonical measures}.
	
The mutual energy $(\cdot, \cdot)_v$ of \eqref{archimedean pairing} and \eqref{non-archimedean pairing} can also be defined for discrete measures.  If $F = \{x_1, \ldots, x_n\}$ is any finite set in a number field $K$, and $v \in M_K$, then denote by $[F]_v$ the probability measure supported equally on the elements of $F\subset \C_v$.  Then 
\begin{equation} \label{finite sets}
	\sum_v\; r_v\; ([F]_v, [F]_v)_v = \sum_v \; r_v \; \frac{1}{|F|^2} \sum_{i \not= j} \log|x_i - x_j|_v = 0
\end{equation}
by the product formula.  

\subsection{A metric on the space of adelic heights}  \label{metric definition}
The height pairing gives rise to a metric on the space of continuous, semipositive, adelic metrics on $\mathcal{O}_{\P^1}(1)$  \cite[Theorem 1]{Fili:energy}.  Given a number field $K$ and any collection of probability measures $\{\mu_v\}_{v \in M_K}$ on $\P^{1,an}_v$ with continuous potentials for which $\mu_v = \omega_v$ at all but finitely many places (where $\omega_v$ is a point mass supported on the Gauss point), then there is a unique metric on $\mathcal{O}_{\P^1}(1)$ with curvature distributions given by $\{\mu_v\}_{v \in M_K}$, normalized such that its associated height function $h: \P^1(\Qbar) \to \R$ satisfies $h\cdot h = 0$ \cite{FRL:equidistribution}.  The height pairing between any two such heights is computed as 
$$h_1 \cdot h_2 = \frac12 \sum_{v \in M_K} \; r_v \; (\mu_{1,v} - \mu_{2,v}, \; \mu_{1,v} - \mu_{2,v})_v.$$
Fili observed that a distance between $h_1$ and $h_2$ can be defined by 
	$$\dist(h_1, h_2) := \left(h_1 \cdot h_2 \right)^{1/2}. $$
Indeed, we have already seen that $h_1 \cdot h_2 = 0$ if and only if $h_1 = h_2$ because of the non-degeneracy of the mutual energy $(\cdot, \cdot)_v$ at each place.  Furthermore, $\dist(\cdot, \cdot)$ satisfies a triangle inequality: at each place, the mutual energy induces a non-degenerate, symmetric, bilinear form on the vector space of measures of mass 0 with continuous potentials on $\P^{1,an}_v$, and so the triangle inequality for $\dist(\cdot, \cdot)$ follows from an $\ell^2$ triangle inequality.

%%%%%%
\bigskip
\section{Non-archimedean energy}
\label{nonarchsection}

Throughout this section, we fix a number field $K$ and a non-archimedean place $v \in M_K$, and provide a lower bound on the non-archimedean local energy defined in Proposition \ref{local energy}:

\begin{theorem}\label{non-arch energy}
For $t_1, t_2 \in K \setminus \{0, 1\},$ we have
$$  E_v(t_1,t_2)-\frac{4}{3}\log |2|_v\geq \left\{
\begin{array}{ll}
\dfrac{\log^2 |t_1/t_2|_v}{6\log \max\{|t_2|_v, |t_1|_v\}} & \mbox{for } { \min\{|t_2|_v, |t_1|_v\}>1}\\
& \\
\dfrac{\log^2 |t_1/t_2|_v}{-6\log \min\{|t_2|_v, |t_1|_v\}} & \mbox{for } { \max\{|t_2|_v, |t_1|_v\}<1 }\\
& \\
\dfrac{|\log |t_1/t_2|_v|}{6}  & { \mbox{otherwise.}}
\end{array} 
\right.  $$
Equality holds for $v \nmid 2$ with $\min\{|t_1-1|_v, |t_2-1|_v\} \geq 1$.
\end{theorem}

\subsection{Measure and escape rate for $v \nmid 2$} \label{not2}

\begin{prop}\label{non-archi julia 1}
Suppose $t \in K \setminus \{0, 1\}$ and $v\nmid 2$.  Then $f_t$ has good reduction at $v$ if and only if $|t(t-1)|_v = 1$.  If $|t(t-1)|_v \not=1$, then $f_t$ fails to have potential good reduction at $v$, and the canonical measure $\mu_{t,v}$ on $\P^{1,an}_v$ of $f_t$ is the interval measure supported on 
  $$I= \left\{ \begin{array}{ll}  
		[\zeta_{0,1},\zeta_{0,|t|_v}] & \mbox{for } |t|_v>1 \mbox{ or } |t|_v<1, \\ 
		\textup{$[\zeta_{0,1}, \zeta_{1, |1-t|_v}]$}& \mbox{for } |1-t|_v<1.  \end{array} \right. $$
\end{prop}

\begin{proof}
By Proposition \ref{G transformation}, it suffices to treat the cases with $|t|_v > 1$.  By \cite[\S5.1]{FRL:ergodic}, $f_t^{-1}(I) = I$ and the action of $f_t$ on $I$ is by a tent map of degree 2.  That is, 
  $$f_t(\zeta_{0,|t|_v^r})= \left\{ \begin{array}{ll}  
		\zeta_{0,|t|_v^{2r-1}} & \mbox{for } 1/2\leq r\leq 1, \\ 
		\zeta_{0,|t|_v^{1-2r}}& \mbox{for } 0 \leq r\leq 1/2.  \end{array} \right. $$
The proposition follows.		
\end{proof}

We may now compute the local height $\lambda_{t,v}(z)=G_{F_t, v}(z,1)$ on $\A_{v}^{1,an}$, which is locally constant away from the interval $[0, \infty) \subset \mathbb{A}^{1, an}_v$.

\begin{prop}\label{big t escape-rate formula}
Suppose $v\nmid 2$ is non-archimedean and $|t(t-1)|_v\geq 1$. The escape-rate function $G_{F_t, v}$ satisfies 
\begin{equation}\label{escape-rate big t}G_{F_t, v}(z, 1)= \left\{  \begin{array}{ll}
				\log |z|_v  & \mbox{for } |z|_v \geq |t|_v \\
				& \\
				\dfrac{1}{2}\left(\dfrac{\log^2|z|_v}{\log|t|_v}+\log |t|_v \right) & \mbox{for } 1 < |z|_v <|t|_v \\
				& \\
				\dfrac12\log |t|_v & \mbox{for } |z|_v\leq 1  \end{array} \right. 
\end{equation}
for all $z\in \C_v$. 				
\end{prop}

\begin{proof} Let $\lambda$ be the continuous extension of the expression on the right hand side of the formula \eqref{escape-rate big t} to  $\A_v^{1,an}$. By Proposition \ref{non-archi julia 1}, $\mu_{t,v}$ is the interval measure corresponding to $[\zeta_{0, 1}, \zeta_{0,|t|_v}]$, and a direct computation shows that 
    $$\Delta \lambda =\mu_{t,v} - \delta_\infty.$$
Thus it suffices to show that $G_{F_t, v}(\cdot, 1)$ and $\lambda$ agree at a single point.  For any $z_0 \in \C_v$ with $|z_0|_v>|t|_v$, define $(z_n, w_n):=F_{t}^n(z_0,1)$, so that
\begin{equation}\label{recursive relation}
(z_{n+1}, w_{n+1})=F_t(z_n,w_n)=((z^2_{n}-tw_{n}^2)^2, 4z_{n}w_{n}(z_{n}-w_{n})(z_{n}-tw_{n})).
\end{equation}
Inductively,
   $$|z_{n}|_v=|z_0|_v^{4^n}>|t|_v|w_{n}|_v>|w_n|_v.$$
Consequently, 
   $$G_{F_t, v}(z_0,1)=\lim_{n\to \infty} \frac{1}{4^n}\log \|F_t^n(z_0, 1)\|_v=\log|z_0|_v=\lambda(z_0).$$
\end{proof}
 
A similar application of Proposition \ref{non-archi julia 1} yields
\begin{prop}\label{small t escape-rate formula}
Suppose $v\nmid 2$ is non-archimedean and $|t|_v<1$. The escape-rate function $G_{F_t, v}$ satisfies 
\begin{equation}G_{F_t, v}(z, 1)= \left\{  \begin{array}{ll}
				\log |z|_v  & \mbox{for } |z|_v \geq 1 \\
				&\\
				-\dfrac{\log^2|z|_v}{2\log|t|_v}+\log |z|_v & \mbox{for } |t|_v < |z|_v <1 \\
				&\\
				\dfrac12\log |t|_v & \mbox{for } |z|_v\leq |t|_v  \end{array} \right.
\end{equation}
for all $z\in \C_v$. 				
\end{prop}

\subsection{Proof of Theorem \ref{non-arch energy} for $v\nmid 2$} 
We compute the local energy $E_v(t_1, t_2)$ by cases.

\medskip
\noindent{\bf Case (1):} $|t_1|_v>1$ and $|t_2|_v<1$. Recall the local energy can be expressed as
   $$2\, E_{v}(t_1, t_2)=\int_{\P^{1,an}_v} \left(\lambda_{t_1, v} - \lambda_{t_2,v}\right) \, d\mu_{t_2,v} + \int_{\P^{1,an}_v} \left( \lambda_{t_2, v} - \lambda_{t_1, v} \right) \, d\mu_{t_1,v}.$$
Therefore by Proposition \ref{non-archi julia 1}, \ref{big t escape-rate formula} and \ref{small t escape-rate formula}, 
\[\begin{split}2E_{v}(t_1, t_2) &= \int_{\P^{1,an}_v} \left(\lambda_{t_1, v} - \lambda_{t_2,v}\right) \, d\mu_{t_2,v} + \int_{\P^{1,an}_v} \left( \lambda_{t_2, v} - \lambda_{t_1, v} \right) \, d\mu_{t_1,v}. \\
&=  \int^0_{\log|t_2|_v} \left(\frac{\log|t_1|_v}{2} -\left(-\frac{x^2}{2\log|t_2|_v}+x\right)\right) \, \frac{dx}{-\log|t_2|_v}\\
&\, \, \, \, \, \, \, + \int^{\log|t_1|_v}_0 \left(x-\frac{1}{2}\left(\frac{x^2}{\log|t_1|_v}+\log|t_1|_v\right)\right) \, \frac{dx}{\log|t_1|_v}\\
&=\frac{\log |t_1/t_2|_v}{3}.
\end{split}
\]

\medskip
\noindent{\bf Case (2):} $|t_1|_v> 1$ and $|t_2|_v> 1$. Without loss of generality, we assume that $|t_1|_v=\max\{|t_1|_v, |t_2|_v\}$. By Proposition \ref{non-archi julia 1} and \ref{big t escape-rate formula}, 
\[\begin{split}2E_{v}(t_1, t_2) &= \int_{\P^{1,an}_v} \left(\lambda_{t_1, v} - \lambda_{t_2,v}\right) \, d\mu_{t_2,v} + \int_{\P^{1,an}_v} \left( \lambda_{t_2, v} - \lambda_{t_1, v} \right) \, d\mu_{t_1,v}. \\
&=  \int_0^{\log|t_2|_v} \left(\frac{1}{2}\left(\frac{x^2}{\log|t_1|_v}+\log|t_1|_v\right) -\frac{1}{2}\left(\frac{x^2}{\log|t_2|_v}+\log|t_2|_v\right)\right) \, \frac{dx}{\log|t_2|_v}\\
&\, \, \, \, \, \, +\int_0^{\log|t_2|_v} \left(\frac{1}{2}\left(\frac{x^2}{\log|t_2|_v}+\log|t_2|_v\right) -\frac{1}{2}\left(\frac{x^2}{\log|t_1|_v}+\log|t_1|_v\right)\right) \, \frac{dx}{\log|t_1|_v}\\
&\, \, \, \, \, \,  + \int^{\log|t_1|_v}_{\log|t_2|_v} \left(x-\frac{1}{2}\left(\frac{x^2}{\log|t_1|_v}+\log|t_1|_v\right)\right) \, \frac{dx}{\log|t_1|_v}\\
&=\frac{\log^2 |t_1/t_2|_v}{3\log \max\{|t_1|_v, |t_2|_v\}}.
\end{split}
\]

\medskip
\noindent{\bf Case (3):} $|t_2(t_2-1)|_v=1$ and $|t_1-1|_v\geq1$. In this case, $f_{t_2}$ has good reduction, so $\mu_{t_2,v}$ is a point mass supported on the Gauss point $\zeta_{0,1}$.  Hence
   $$2E_{v}(t_1, t_2)=\frac{|\log |t_1|_v|}{3}=\frac{|\log |t_1/t_2|_v|}{3}.$$
   
\medskip 
\noindent{\bf Case (4):} The remaining cases reduce to the above three by the symmetry relations of Proposition \ref{one sing}.  This completes the proof of Theorem \ref{non-arch energy} under the assumption that $v \nmid 2$.

\subsection{Measure and escape rate for $v \mid 2$}

\begin{prop}\label{nonarch julia 2}
Suppose $v\mid 2$ is non-archimedean. The canonical measure $\mu_{t,v}$ on $\P^{1,an}_v$ of $f_t$ is the interval measure corresponding to the interval $I$ with
  $$I= \left\{ \begin{array}{ll}  
		[\zeta_{0,|t/4|_v},\zeta_{0,|4|_v}] & \mbox{for } |t|_v < |16|_v, \\ 
		\textup{$[\zeta_{0,|1/4|_v}, \zeta_{0, |4t|_v}]$}& \mbox{for } |t|_v > 1/|16|_v,\\ 
		\textup{$[\zeta_{1,|1-t|_v/|4|_v}, \zeta_{1,|4|_v}]$}  & \mbox{for } |1-t|_v <  |16|_v.\end{array} \right.$$
For $|16|_v\leq |t|_v\leq 1/|16|_v$ with $|1-t|_v\geq |16|_v$, $f_t(z)$ has potential good reduction, and $\mu_{t,v}$ is supported on a single point in $\Hyp$. 
\end{prop}

\begin{proof} 
We proceed as in the computations of \cite[\S5.1]{FRL:ergodic}, though the authors had assumed for simplicity that the residue characteristic of their field is not 2.  If $|t|_v > |1/16|_v$, the interval $[\zeta_{0,|1/4|_v}, \zeta_{0, |4t|_v}]$ is totally invariant by $f_t$, and 
$$f_t(\zeta_{0,|4t|_v|16t|_v^{-r}})=\zeta_{0,|4t|_v|16t|_v^{-2r}} \; \mbox{ and } \; f_t(\zeta_{0,|4t|_v|16t|_v^{r-1}})=\zeta_{0,|4t|_v|16t|_v^{-2r}}$$
for $r\in [0,1/2]$.  Thus $\mu_{t,v}$ is the interval measure on $[\zeta_{0,|1/4|_v}, \zeta_{0, |4t|_v}]$. The cases $|t|_v < |16|_v$ or $|1-t|_v < |16|_v$ can then be deduced from Proposition \ref{G transformation}.

For all $|16|_v\leq |t|_v\leq 1/|16|_v$ with $|1-t|_v\geq |16|_v$, we have $|j(t)|_v\leq 1$, so $f_t$ has potential good reduction.
\end{proof}

Following the proofs of Propositions \ref{big t escape-rate formula} and \ref{small t escape-rate formula}, from Proposition \ref{nonarch julia 2} we obtain

\begin{prop}\label{escape rate v|2}
Suppose $v\mid 2$ is non-archimedean.  We have
\begin{equation}
G_{F_t, v}(z, 1)= \left\{  \begin{array}{ll}
				\log |z|_v  & \mbox{for } |z|_v \geq |4t|_v \\
				&\\
				\dfrac{1}{2}\left(\dfrac{\log^2|4z|_v}{\log|16t|_v}+\log |t|_v\right) & \mbox{for } 1/|4|_v< |z|_v<|4t|_v  \\
				&\\
				\dfrac{1}{2}\log |t|_v & \mbox{for } |z|_v\leq1/|4|_v\end{array} \right.
\end{equation}
for $t$ with $|t|_v\geq 1/|16|_v$, and 
\begin{equation}G_{F_t, v}(z, 1)= \left\{  \begin{array}{ll}
				\log |z|_v  & \mbox{for } |z|_v \geq |4|_v \\
				&\\
				\dfrac{1}{2}\left(\dfrac{\log^2|4z/t|_v}{\log|16/t|_v}+\log |t|_v\right) & \mbox{for } |4|_v< |z|_v<|t/4|_v  \\
				&\\
				\dfrac{1}{2}\log |t|_v & \mbox{for } |z|_v\leq |t/4|_v\end{array} \right.
\end{equation}
for $t$ with $|t|_v\leq |16|_v$.
\end{prop}

\subsection{Proof of Theorem \ref{non-arch energy} for $v\mid 2$}

We compute as in the case where $v\nmid 2$.

\medskip
\noindent{\bf Case (1):} $\{t_1, t_2\}$ with $\min\{|t_1|_v, |t_2|_v\}\geq 1/|16|_v$ and $\max\{|t_1|_v, |t_2|_v\}> 1/|16|_v$. Proposition \ref{escape rate v|2} yields
  $$2\, E_v(t_1,t_2)=\frac{\log^2|t_1/t_2|_v}{3\log \max\{|16t_1|_v, |16t_2|_v\}}\geq \frac{\log^2|t_1/t_2|_v}{3\log \max\{|t_1|_v, |t_2|_v\}}.$$
  
 \medskip
\noindent{\bf Case (2):} $|t_1|_v\geq 1/|16|_v$ and $|t_2|_v\leq |16|_v$.  Again by Proposition \ref{escape rate v|2},
$$2\, E_v(t_1,t_2)=\frac{\log|16t_1|_v-\log |t_2/16|_v}{3}-\log |16|_v\geq \frac{\log|t_1/t_2|_v}{3}.$$

\medskip
\noindent{\bf Case (3):} $|t_1|_v> 1/|16|_v$, $|16|_v\leq |t_2|_v\leq 1/|16|_v$ and $|1-t_2|_v\geq|16|_v$.  Let $\zeta_{t_2}\in \Hyp$ be the support of $\mu_{t_2, v}$. For any $z\in \C_v$ with $|z|_v>1/|4|_v$, 
   $$|f_{t_2}(z)|_v=\frac{|(z^2-t_2)^2|_v}{|4z(z-1)(z-t_2)|_v}>|z|_v.$$
Hence $\zeta_{0, 1/|4|_v}\in [\zeta_{t_2}, \infty)$. Let $z_0\in \C_v$ with $|z_0|_v>1/|4|_v$, and let $(z_n, w_n):=F_{t_2}^n(z_0, 1)$.
From the recursive formula \eqref{recursive relation}, inductively we have $|z_n|=|z_0|_v^{4^n}>|w_n|_v/|4|_v$. Consequently  
  $$\lambda_{t_2, v}(z)=G_{F_{t_2}, v}(z,1)=\lim_{n\to \infty}\frac{\log \|F^n_{t_2}\|_v}{4^n}=\log |z|_v$$
for $z$ with $|z|_v>1/|4|_v$, and then $\lambda_{t_2,v}(\zeta_{0, r})=\log r$ for $r\geq 1/|4|_v$. Moreover, as $\Delta \lambda_{t_2, v}=\delta_{\zeta_{t_2}} - \delta_{\infty}$, the function $\lambda_{t_2,v}$ is increasing at a constant rate along the ray $[\zeta_{t_2}, \infty)$, with respect to the hyperbolic metric. Therefore $\lambda_{t_2, v}(\zeta_{t_2})\leq \lambda_{t_2,v}(\zeta_{0, 1/|4|_v})=-\log|4|_v$. Hence by Propositions \ref{nonarch julia 2} and \ref{escape rate v|2}, 
\begin{eqnarray*} 
2\, E_v(t_1, t_2) &=&  \int_{\P^{1,an}_v} \left(\lambda_{t_1, v} - \lambda_{t_2,v}\right) \, d\mu_{t_2,v} + \int_{\P^{1,an}_v} \left( \lambda_{t_2, v} - \lambda_{t_1, v} \right) \, d\mu_{t_1,v}. \\
&=&  \left(\lambda_{t_1,v}(\zeta_{t_2}) -\lambda_{t_2, v}(\zeta_{t_2})\right) \, \\
&& + \int^{\log|4t_1|_v}_{\log|1/4|_v} \left(x-\frac{1}{2}\left(\frac{(x+\log|4|_v)^2}{\log|16t_1|_v}+\log|t_1|_v\right)\right) \, \frac{dx}{\log|16t_1|_v}\\
&\geq& \frac{\log |16t_1|_v}{3}.
\end{eqnarray*}
Here we have used $\lambda_{t_2, v}(\zeta_{t_2})\leq-\log|4|_v$ and $\lambda_{t_1,v}(\zeta_{t_2})=\frac{1}{2}\log|t_1|_v$ for the last inequality. 
   
Of course, for  $|16|_v\leq |t_i|_v\leq1/|16|_v$ and $|1-t_i|_v\geq |16|_v$ for $i=1,2$, we have 
   $$2\, E_v(t_1, t_2)\; \geq\;  0\; \geq \; \frac{|\log |t_1/t_2|_v|}{3}+\frac{8}{3}\log|2|_v.$$
\medskip \noindent{\bf Case (4):} The remaining cases reduce to the above three by the symmetry relations of Proposition \ref{one sing}. This completes the proof of Theorem \ref{non-arch energy}.

%%%%%%
\bigskip
\section{Archimedean places and the hybrid space} \label{archsection}

In this section, we provide some of the estimates we need to control the archimedean contributions to the height pairings.  Throughout this section, we assume our parameter $t \in \C\setminus\{0,1\}$ is complex.  We let $\mu_t$ denote the probability measure on $\P^1(\C)$ which is the push forward of the Haar measure on the Legendre elliptic curve $E_t(\C)$ via $\pi(x,y) = x$.  This measure is also the unique measure of maximal entropy for the dynamical system defined by the Latt\`es map 
	  $$f_t(z)=\frac{(z^2-t)^2}{4z(z-1)(z-t)},$$
as noted in \cite[\S7]{Milnor:Lattes}.
We study degenerations of the probability measures $\mu_t$ and their potentials as $t\to 0$.  (The cases of $t\to 1$ and $t\to \infty$ are similar.)  To this end, we consider the action of $f_t$ sending $(t,z)$ to $(t, f_t(z))$ on the complex surface $X = \D^* \times \P^1(\C)$, where $\D^*$ is the punctured unit disk.  We make use of the hybrid space $X^{hyb}$, in which the Berkovich projective line over the field of formal Laurent series $\C((t))$ creates a central fiber of $X$ over $t=0$ in the unit disk $\D$.  We appeal to the topological description of the hybrid space from \cite{Boucksom:Jonsson:hybrid} and the associated dynamical degenerations described in \cite{Favre:degenerations}.  

\subsection{The family of Latt\`es maps and their escape rates}  \label{Berkovich limit}
In homogeneous coordinates on $\C^2$, recall that the maps $f_t$ may be presented as 
  $$F_t(z,w):=\left( (z^2-tw^2)^2, 4 zw(z-w)(z-tw) \right),$$
for $t\in \C\setminus\{0,1\}$.  They have escape-rate functions
\begin{equation} \label{arch escape rate}
  G_{F_t}(z,w):=\lim_{t\to \infty}\frac{1}{4^n} \log\|F_t^n(z,w)\|,
\end{equation}
as in (\ref{escape rate}).

View the families $f_t$ and $F_t$ as maps $f = f_T$ and $F = F_T$ defined over the field $k = \C(T)$, and consider the non-archimedean absolute value $|\cdot|_0$ on $k$ satisfying $|g(T)|_0 = e^{-\ord_0 g}$.  Let $k_0 = \C((T))$ denote the completion of $\C(T)$ with respect to this absolute value.  Let $\mathbb{L}$ denote a (minimal) complete and algebraically closed field containing $k_0$.  The non-archimedean escape rate $\hat{G}_F$ on $\mathbb{L}^2$ is defined as in \eqref{escape rate}.  Since $|T|_0 < 1,$ it is given for $x \in \mathbb{L}$ by the following formula, exactly as in Proposition \ref{small t escape-rate formula}:
\begin{eqnarray} \label{g hat}
\hat{g}_f(x) := \hat{G}_F(x, 1) &=&  \left\{ \begin{array}{ll}
				\log|x|_0 & \mbox{for }  |x|_0 \geq 1 \\
				\log|x|_0 - \dfrac{(\log |x|_0)^2}{2\log|T|_0}  &  \mbox{for } |T|_0 <  |x|_0 <1  \\
				\dfrac12 \log |T|_0 & \mbox{for } |x|_0 \leq  |T|_0
				\end{array} \right.  \\
	&=& \left\{ \begin{array}{ll}
				-a & \mbox{for }  |x|_0 = |T|^a_0 \mbox{ with } a \leq 0  \\
				-a + \dfrac12 \, a^2  &  \mbox{for } |x|_0 = |T|^a_0 \mbox{ with }  0 \leq a \leq 1  \\
				-\dfrac12  & \mbox{for }  |x|_0  = |T|^a_0 \mbox{ with } a \geq 1
				\end{array} \right.   \nonumber
\end{eqnarray}
The function $\hat{g}_f$ extends naturally to the Berkovich space ${\P}^{1, an}_{\mathbb{L}}$; away from the point at $\infty$, it is a continuous potential for the equilibrium measure $\hat{\mu}_f$ of $f$.  

The potential $\hat{g}_f$ and the measure $\hat\mu_f$ are invariant under the action of $\Gal(\mathbb{L}/k_0)$ on $\P^{1,an}_{\mathbb{L}}$.  They descend to define a function and probability measure  -- that we will also denote by $\hat{g}_f$ and $\hat\mu_f$ -- on the quotient Berkovich line ${\P}^{1, an}_{k_0}$ (see \cite[\S4.2]{Berkovichbook} for details on this quotient map).  As computed in Proposition \ref{non-archi julia 1}, the measure $\hat{\mu}_f$ is supported on the interval $[\zeta_{0, |T|_0}, \zeta_{0,1}]$, and it is uniform with respect the linear structure from the hyperbolic metric.

\subsection{Convergence of measures} \label{measures}
The family $f_t$ acts on the product space $\D^*\times \P^1$ sending $(t,z)$ to $(t, f_t(z))$.  It extends meromorphically to $X_0 := \D\times\P^1$, or indeed to any model complex surface $X \to \D$ which is isomorphic to $\D^*\times\P^1$ over $\D^*$ and has a simple normal crossings divisor as its central fiber.  

Fixing a surface $X \to \D$ and letting $t\to 0$, the degeneration of the measures $\mu_t$ of maximal entropy for $f_t$ -- or indeed for any meromorphic family of rational maps on $\P^1$ -- to the central fiber of $X$ is now well understood.  In \cite{DF:degenerations, DF:degenerations2}, the limit of the measures $\mu_t$ is computed for any choice of model $X$, and a relation is shown between these limits and the non-archimedean measure $\hat{\mu}_f$.  In particular, if we define the annulus
	$$A_t(a,b,C) := \{z \in \C:  C^{-1} |t|^a \leq |z| \leq C |t|^b \}$$
for $t\in \D^*$, $C>1$, and real numbers $a \geq b$, then 
\begin{equation} \label{measures on annuli}
	\mu_t(A_t(a,b,C)) \to \hat\mu_f([\zeta_{0, |T|_0^a}, \zeta_{0, |T|_0^b}]) = \mathrm{length}_{\R}([0,1]\cap[b,a])
\end{equation}
as $t\to 0$.  This follows from \cite[Theorem B]{DF:degenerations} (allowing for changes of coordinates on $\P^1$ and base changes, passing to covers of the punctured disk $\D^*$) or from the computations described in \cite[Theorem D]{DF:degenerations2} (taking $\Gamma$ to be a vertex set in the interval $[\zeta_{0,1}, \zeta_{0, |T|_0}]$).  Another proof is described below in \S\ref{t near 0}.  In particular, this convergence implies:

\begin{lemma} \label{measure decomposition}
Given any $\eps>0$ and integer $n\geq 1$, there exists $\delta>0$ such that 
	$$\frac{1}{n} - \eps < \mu_t( \{ |t|^{(i+1)/n} \leq |z| \leq |t|^{i/n} \}) < \frac{1}{n} + \eps$$
for all $0 < |t| < \delta$ and $i = 0, \ldots, n-1$.
\end{lemma}

Taking $\eps = 1/n^2$ in Lemma \ref{measure decomposition}, we observe that for any given $n$, there is a $\delta>0$ such that we also have
\begin{equation} \label{small measure disks}
	\mu_t\bigg( \{|z|\geq 1\} \cup \{|z| \leq |t| \} \bigg) < \frac{1}{n}
\end{equation}
for all $0 < |t| < \delta$.

\subsection{Convergence in the hybrid space} \label{t near 0}

In \cite{Favre:degenerations}, Favre gives an alternate proof of \eqref{measures on annuli} by showing that 
\begin{equation} \label{measure convergence}
	\mu_t \to \hat\mu_f
\end{equation}
weakly in the hybrid space $X^{hyb}$ \cite[Theorem B]{Favre:degenerations}.  The hybrid space consists of replacing the central fiber in the models $X$ above with the Berkovich line $\P^{1, an}_{k_0}$, carrying an appropriate topology.  The convergence of measures follows from the convergence of their potentials to the potential of the measure $\hat\mu_f$ in the Berkovich line.  We describe this convergence here, as we will use it for proving our main result.

Let $m_1$ denote the Lebesgue measure on the unit circle in $\C$, normalized to have total length 1.  Let $\Phi_t(z)$ denote a continuous potential on  $\P^1(\C)$ for the measure $\mu_t - m_1$.  Explicitly, in local coordinates $z\in\C \subset \P^1$, we can take
\begin{equation} \label{Phi}
	\Phi_t(z) = G_{F_t}(z, 1) - \log^+|z|
\end{equation}
with $G_{F_t}$ as in \eqref{arch escape rate}.  
In \cite{Favre:degenerations}, Favre proves that the function 
\begin{equation} \label{continuous potential}
	\phi(t, z) := \frac{\Phi_t(z)}{\log |t|^{-1}}
\end{equation}
extends to define a continuous function on $X^{hyb}$, taking the values of a potential of the limiting measure $\hat\mu_f - \omega_0$ on the central fiber.  Here $\omega_0$ is the delta mass on the Gauss point $\zeta_{0,1}$ of the Berkovich line $\P^{1,an}_{k_0}$.  More precisely, we consider the function
\begin{eqnarray} \label{phi hat}
\hat{\phi}_f(x) &:=&  \left\{ \begin{array}{ll}
				0 & \mbox{for }  |x|_0 \geq 1 \\
				\log|x|_0 - \dfrac{(\log |x|_0)^2}{2\log|T|_0}  &  \mbox{for } |T|_0 <  |x|_0 <1  \\
				\dfrac12 \log |T|_0 & \mbox{for } |x|_0 \leq  |T|_0
				\end{array} \right.  \\
	&=& \left\{ \begin{array}{ll}
				0 & \mbox{for }  |x|_0  = |T|^a_0 \mbox{ with } a \leq 0  \\
				-a + \dfrac12 \, a^2  &  \mbox{for } |x|_0 = |T|_0^a \mbox{ with }  0 \leq a \leq 1  \\
				-\dfrac12  & \mbox{for } |x|_0  = |T|_0^a \mbox{ with } a \geq 1
				\end{array} \right.   \nonumber
\end{eqnarray}
for $x \in \mathbb{L}$, similar to the formula for $\hat{g}_f$ in \eqref{g hat}.  This function $\hat\phi_f$ extends continuously to all of $\P^{1, an}_{\mathbb{L}}$; it is Galois invariant over $k_0$; and it descends to the quotient $\P^{1, an}_{k_0}$.  Favre's theorem implies that the function $\phi$ of \eqref{continuous potential} extends continuously to $X^{hyb}$, coinciding with $\hat{\phi}_f$ over $t=0$:

\begin{prop} \label{uniform continuity}
Given any $\eps>0$, there exists $\delta >0$, such that 
	$$\left| \phi(t,z) - \hat{\phi}_f(\zeta_{0, |T|^a_0}) \right| < \eps$$
for all $0 < |t| < \delta$, for all $a \in \R$, and all $z$ for which 
	$$\left| \frac{\log|z|}{\log|t|} - a \right| < \delta.$$
\end{prop}

\begin{proof}    
Recall that the absolute value $|\cdot |_0$ on $\mathbb{L}$ induces a continuous function on the Berkovich space that we will also denote by $|\cdot |_0 :  \P^{1,an}_{k_0} \to \R_{\geq 0}\cup\{\infty\}$.  We use the standard absolute value $|\cdot|$ on $\C$, extended to a continuous function $\P^1(\C) \to \R_{\geq 0} \cup \{\infty\}$.

The topology on $\P^{1, an}_{k_0}$ is such that annuli of the form
	$$A(r_1, r_2) := \{x \in \P^{1,an}_{k_0} :  r_1 < |x|_0 < r_2\}$$
are open for any choice of $0\leq r_1 < r_2 \leq \infty$, as are the Berkovich disks of the form 
	$$D_0(r) :=  \{x \in \P^{1,an}_{k_0} :  |x|_0 < r\} \mbox{ and } D_\infty(r) := \{x \in \P^{1,an}_{k_0} :  |x|_0 > r\}$$
for any $0 < r < \infty$.  The topology on $X^{hyb}$ is such that an annular set of the form
	$$\{(t,z) \in \D^*\times\P^1(\C):  |t|^{a+\delta} < |z| < |t|^{a-\delta} \mbox{ and } 0 < |t| < \delta\} \cup A(|T|^{a+\delta}_0, |T|^{a-\delta}_0)$$
is an open neighborhood of $\zeta_{0,|T|_0^a}$ on the central fiber for any $a$ and any $\delta > 0$.  Similarly, the disk-like sets 
	$$\{(t,z) \in \D^*\times\P^1(\C):  |z| < |t|^a \mbox{ and } 0 < |t| < \delta\} \cup D_0( |T|^a_0)$$
and
	$$\{(t,z) \in \D^*\times\P^1(\C):  |z| > |t|^a \mbox{ and } 0 < |t| < \delta\} \cup D_\infty(|T|^a_0)$$
are open for any $a\in \R$, and allowing $a$ and $\delta$ to vary provides open neighborhoods at $0$ and $\infty$ respectively in the central fiber.  See \cite[\S2.2 and Definition 4.9]{Boucksom:Jonsson:hybrid} for details on the hybrid topology.  Note in particular that the hybrid topology restricted to the central fiber induces the usual (weak) Berkovich topology.
	
By the continuity statement of \cite[Theorem 2.10]{Favre:degenerations} and exhibiting $\phi$ as a uniform limit of model functions (\cite[Section 4.3]{Favre:degenerations} provides the details in the dynamical case) the function $\phi$ extends to define a continuous function on $X^{hyb}$, taking the values of $\hat{\phi}_f$ on the central fiber.  Let $L$ denote the closed segment in $\mathbb{P}^{1,an}_{k_0}$ between $0$ and $\infty$.  We may by compactness cover $L$ by finitely many neighborhoods on which $|\phi(x) - \phi(y)| \leq \epsilon.$  As the values of $\hat{\phi}_f$ depend only on the values of $\phi$ on $L$, each open neighborhood of a point in the interior of $L$ contains an open interval in $L$, and $\hat{\phi}$ is constant near $0$ and $\infty$, we may assume these neighborhoods are annular or disk-like as defined above.  Thus we obtain a uniform $\delta$ as claimed.
\end{proof}

As $\Phi_t(z) = G_{F_t}(z,1) - \log^+|z|$, we also have a uniform continuity statement for $G$ when $|z|$ is bounded from above:

\begin{prop} \label{uniform continuity for G}
Given any $\eps>0$ and $M > 1$, there exists $\delta >0$ such that 
	$$\left| \frac{G_{F_t}(z,1)}{\log|t|^{-1}} - \hat{g}_f(\zeta_{0, |T|^a_0}) \right| < \eps$$
for all $0 < |t| < \delta$, for all $a \in \R$, and all $|z|\leq M$ for which 
	$$\left| \frac{\log|z|}{\log|t|} - a \right| < \delta.$$
\end{prop}

\subsection{Discrete measures and regularizations}
Let $F$ be any finite set in $\C$.  Denote by $[F]$ the probability measure supported equally on the elements of $F$, and for $r>0$, denote by $[F]_r$ the probability measure supported equally and uniformly on circles of radius $r$ about each element of $F$.  

\begin{prop}\label{arch regularization}
For every $\eps>0$, there exists $c = c(\eps) >0$ such that
	$$\left| (\mu_t, [F]) - (\mu_t, [F]_r) \right| \; < \; \eps \; \max\{ \log |t|^{-1}, \log |t-1|^{-1}, \log|t|, 1 \}$$
for all $t\in \C\setminus\{0,1\}$ and any finite set $F$ in $\C$ and any
	$$r \leq c \, \min\{|t|^2, |t-1|^2, |t|^{-2}\}.$$
\end{prop}

\begin{proof}
For any $x \in \C$ and any $r>0$, let $m_{x,r}$ be the probability measure supported on the circle of radius $r$ around $x$.  Recall that 
$$(\rho, \sigma) := - \iint_{\C\times\C \setminus \Delta} \log |z - w| \ d \rho(z) \ d \sigma(w).$$
For each fixed $t$, the function $G_{F_t}(\cdot, 1)$ is a potential for $\mu_t$ in $\C$, and therefore, there exists a constant $C_t$ such that 
	$$\int_\C \log|z-w| \, d\mu_t(z) = G_{F_t}(w,1) + C_t.$$
Now let $F$ be any finite set in $\C$.   Then, assuming $r < 1$, we have
\begin{eqnarray*}
(\mu_t, [F]_r) - (\mu_t, [F]) 
	&=& \frac{1}{|F|} \sum_{x \in F} \left( G_{F_t}(x, 1) - \int G_{F_t}(\zeta, 1) \, dm_{x,r}(\zeta) \right) \\
	&=& \frac{1}{|F|}  \sum_{\{x \in F: |x| > 2\}} \left( \Phi_t(x) - \int \Phi_t(\zeta) \, dm_{x,r}(\zeta) \right) \\
	&& \; + \quad \frac{1}{|F|}  \sum_{\{x \in F:  |x| \leq 2\}} \left( G_{F_t}(x, 1) - \int G_{F_t}(\zeta, 1) \, dm_{x,r}(\zeta) \right)
\end{eqnarray*}
because the function $\log^+|z|$ is harmonic away from the unit circle on $\C$.  

By Proposition \ref{uniform continuity for G}, there exists $\delta > 0$ such that 
\begin{equation} \label{G limit a}
	\left| \frac{G_{F_t}(z,1)}{\log|t|^{-1}} - \hat{g}_f(\zeta_{0, |T|^a_0}) \right| < \eps/2
\end{equation}
for all $|z| \leq 2$ satisfying 	
	$$\left| \frac{\log|z|}{\log|t|} - a \right| < \delta$$
and all $|t| < \delta$ and any $a\in \R$.  Shrinking $\delta$ if needed, we have 
\begin{equation} \label{Phi limit 0}
	\left| \phi(t,z) \right| < \eps/2
\end{equation}
for $|z| \geq 1$ and all $|t| < \delta$, by Proposition \ref{uniform continuity}.  

Let $C_\delta$ be the compact subset of $\C\setminus\{0,1\}$ consisting of all $t$ with $|t| \geq \delta$ and $|t-1| \geq \delta$ and $|1/t|\geq \delta$.  Over $C_\delta \times \mathbb{P}^1$, the family of potentials $\{\Phi_t\}$ is uniformly continuous.  So there exists $c_1 = c_1(\delta)$ such that 
	$$|\Phi_t(z) - \Phi_t(z')| < \eps$$
whenever $\mathrm{dist}(z,z') < c_1$ and for all $t \in C_\delta$.  Here, $\mathrm{dist}$ represents the chordal distance on $\P^1$.  Furthermore, we may take $c_1$ such that we also have 
	$$|G_{F_t}(z,1) - G_{F_t}(z',1)| < \eps$$
for all $|z-z'| < c_1$ with $|z| \leq 2$, and all $t \in C_\delta$.  Thus
	$$|(\mu_t, [F]_r) - (\mu_t, [F])| < \eps$$
for any choice of finite set $F$, $t\in C_\delta$, and $r < c_1$.  

Now assume that $|t| < \delta$. We will consider three cases.  First, suppose $|t|^{1+\delta} \leq |z| \leq 2$.  Choose any $c_2 = c_2(\delta)$ such that 
\begin{equation} \label{choice of c2}
	\left|  \log(1 \pm c_2 \delta^{1-\delta}) \right| <  \delta (\log \delta^{-1}).
\end{equation}
Then
\begin{eqnarray*}
\left| \log|z'/z| \right| &=& \left| \log \left|  \frac{z'-z}{z} + 1 \right| \right| \\
	&\leq& \max\left\{ \log \left(1 + \left| \frac{z'-z}{z}\right| \right), \left| \log \left(1 - \left| \frac{z'-z}{z}\right| \right) \right| \right\}  \\
	&\leq& \max \left|  \log \left( 1 \pm \frac{ c_2 |t|^2 }{|z|}  \right)  \right| \\
	&\leq&  \max  \left|  \log \left( 1 \pm c_2 \delta^{1-\delta} \right) \right|  \\
	&<&  \delta (\log \delta^{-1}) \; \leq \; \delta \log|t|^{-1}
\end{eqnarray*}
for all $|t| < \delta$. This is equivalent to

\begin{equation} \label{a and a'}
	\left| \frac{\log|z|}{\log|t|} - \frac{\log|z'|}{\log|t|} \right| < \delta
\end{equation}
for all $|z-z'| < c_2 |t|^2$ with $|t|^{1+\delta} \leq |z| \leq 2$.  Combined with \eqref{G limit a} and setting $a = (\log|z|)/(\log|t|)$, this implies that 
	$$|G_{F_t}(z,1) - G_{F_t}(z',1)| < \eps \log|t|^{-1}$$
for such pairs $z$ and $z'$.  

Second, suppose that $|z| \leq |t|^{1+\delta}$.  By shrinking $c_2$ further if necessary, we have $c_2 < (1 - \delta^\delta)/\delta$, and therefore if $|z| \leq |t|^{1+\delta}$ and $|z- z'| < c_2 |t|^2$, with $|t| < \delta$, we also have $|z'| \leq |t|$.  Applying the convergence \eqref{G limit a} where $\hat{g}_f = -1/2$, for all $|z| \leq |t|^{1+\delta}$ we have 
	$$|G_{F_t}(z,1) - G_{F_t}(z',1)| < \eps \log|t|^{-1}$$
for $z'$ satisfying $|z- z'| < c_2 |t|^2$ and for all $|t| < \delta$.  

Third, for $|z|\geq 2$, by the convergence \eqref{Phi limit 0},
	$$|\Phi_t(z) - \Phi_t(z')| < \eps \log|t|^{-1}$$
for all $|z| \geq 2$ and $|z-z'| < c_2 |t|^2$ and $|t|< \delta$.  

Together these three cases yield 
	$$|(\mu_t, [F]_r) - (\mu_t, [F])| < \eps \log|t|^{-1}$$
for any choice of finite set $F$ and all $|t| < \delta$, with $r < c_2 |t|^2$.  

If $|t-1| < \delta$, the arguments above go through by replacing $z$ with $1-z$, as
		$$G_{F_{1-t}}(1-z, 1) = G_{F_t}(z,1)$$
by Proposition \ref{G transformation}.  It follows that 
	$$|(\mu_t, [F]_r) - (\mu_t, [F])| < \eps \log|t-1|^{-1}$$
for any choice of finite set $F$, $|t-1| < \delta$, and $r < c_2 \,  |t-1|^2$, with $\delta$ and $c_2$ as above.

For $t$ near $\infty$, more care is needed, as
	$$G_{F_t}(z,1) =  G_{F_{1/t}}(z, t) = G_{F_{1/t}}(z/t, 1) + \log|t|$$
by Proposition \ref{G transformation}.   Setting $s = 1/t$,
	$$\frac{G_{F_t}(z,1)}{\log|t|} = \frac{G_{F_s}(s\, z, 1) - \log |s|}{\log |s|^{-1}} = \frac{G_{F_s}(s \, z, 1)}{\log |s|^{-1}} + 1.$$
From \eqref{G limit a}, we have 
	$$	\left| \frac{G_{F_s}(s\,z,1)}{\log|s|^{-1}} - \hat{g}_f(\zeta_{0, |T|^{1-a}_0}) \right| < \eps/2 $$
for $|sz| \leq 2$, $|s| < \delta$, and
	$$\left| \frac{\log|sz|}{\log|s|} - (1-a) \right| < \delta,$$
for any choice of $a\in \R$.  Thus,
\begin{equation} \label{G limit infinity}
	\left| \frac{G_{F_t}(z,1)}{\log|t|} - \left(\hat{g}_f(\zeta_{0, |T|^{1-a}_0}) + 1\right) \right| < \eps/2
\end{equation}
for all $|z| \leq 2|t|$ satisfying 	
	$$\left| \frac{\log|z|}{\log|t|} - a \right| < \delta$$
with $|t| > \delta^{-1}$ and any $a\in \R$.  As in \eqref{Phi limit 0}, we also have 
	$$	\left| \phi(t,z) \right| < \eps/2 $$
for $|z| \geq |t|$ and $|t| > \delta^{-1}$, because $\hat{g}_f(\zeta_{0, |T|^{1-a}_0}) + 1 = a$ for all $a\geq 1$, from the formula given in \eqref{g hat}.  The choice of $c_2$ in \eqref{choice of c2} is similar.    It follows that 
	$$|(\mu_t, [F]_r) - (\mu_t, [F])| < \eps \log|t|$$
for any choice of finite set $F$ and all $|t| > 1/\delta$, with $r < c_2 \,  |t|^{-2}$.  

Let $c := \min \{ c_1, c_2 \}$ to complete the proof.
\end{proof}

%%%%%%
\bigskip
\section{Archimedean energy}
\label{arch energies}

As in Section \ref{archsection}, assume $t \in \C\setminus\{0,1\}$ is a complex parameter, with $\mu_t$ on $\P^1(\C)$ the push-forward of the Haar measure on $E_t(\C)$, and $\lambda_t(z) = G_{F_t}(z,1)$ a potential for $\mu_t - \delta_\infty$ on $\P^1(\C)$.  In this section we provide estimates on the archimedean local energy (introduced in Proposition \ref{local energy})
	$$E_\infty(s,t) := \frac12 \left( \int ( \lambda_s - \lambda_t)  \, d\mu_{t} + \int (\lambda_t - \lambda_s) \, d\mu_{s} \right),$$
for $s, t \in \C\setminus \{0,1\}$ as one or both of the parameters tends to 0, 1, or $\infty$.    We treat three cases separately:  where only one parameter escapes into a cusp, where both parameters escape into a cusp, and where the two parameters head to two different cusps.  By the symmetry established in Proposition \ref{one sing}, we focus on the case where $s$ tends to 0.  

Throughout, we work in hybrid space and make use of the convergence of potentials to $\hat{g}_f$ and measures to $\hat{\mu}_f$ as $t\to 0$, as proved in \cite[Theorem B]{Favre:degenerations} and explained in \S\ref{t near 0}.

\subsection{A single escaping parameter}

\begin{theorem} \label{s to 0}
Given $\eps>0$ and any compact set $C \subset \C\setminus\{0,1\}$, there exists $\delta>0$ such that
$$\left( \frac{1}{6} - \eps \right)  \log |s|^{-1} \leq E_\infty(s,t) \leq \left( \frac{1}{6}  + \eps \right)  \log |s|^{-1}$$
for all $s$ satisfying $0 <  |s|  < \delta$ and all $t\in C$.
\end{theorem}

\begin{proof}
Recall that for any $s \in \C\setminus\{0,1\}$,  we have defined $\Phi_s(z) = G_{F_s}(z,1) - \log^+|z|$ in \eqref{Phi} and $\phi(s,z) = \Phi_s(z)/(\log|s|^{-1})$.  For any pair $s,t \in \C\setminus\{0,1\}$, the local energy $E_\infty(s,t)$ satisfies
\begin{eqnarray*}
\frac{E_\infty(s, t) }{\log|s|^{-1}} 
	&=&\frac{1}{2\log|s|^{-1}} \left( \int (\Phi_{s} - \Phi_{t}) \, d\mu_{t} + \int (\Phi_{t} - \Phi_{s}) \, d\mu_{s} \right) \\
	&=& \frac12 \left( \int \left(\phi(s,z) - \frac{\Phi_t}{\log|s|^{-1}} \right) \, d\mu_{t} + \int \frac{\Phi_t}{\log|s|^{-1}} \ d\mu_{s} - \int \phi(s,z)  \, d\mu_{s} \right) .
\end{eqnarray*}

Fix $\eps>0$ and suppose that $C \subset \C\setminus\{0,1\}$ is compact.  The $\Phi_t$ functions are uniformly bounded for all $t \in C$ and all $z \in \P^1(\C)$, so there is a $\delta$ such that 
$$\left| \int \frac{\Phi_t}{\log|s|^{-1}} \, d\mu_s \right| < \eps$$
for all $|s| < \delta$ and all $t\in C$.  We can also find a small $r = r(C)$ such that 
	$$\mu_t(\{|z| \leq r\}) < \eps$$
for all $t \in C$.  By Proposition \ref{uniform continuity}, (shrinking $\delta$ if needed) 
	$$|\phi(s,z)| < \eps$$
for all $|z| > r$ and $|s| < \delta$, and 
	$$|\phi(s,z)| < 1$$
for all $z$ and all $|s| < \delta$.  Consequently, 
\begin{eqnarray*}
\left| \int \left(\phi(s,z) - \frac{\Phi_t}{\log|s|^{-1}} \right) \, d\mu_{t} \right|  &\leq& \int_{\{|z|\leq r\}} |\phi(s,z)| \,d\mu_t + \int_{\{|z|> r\}} |\phi(s,z)| \,d\mu_t  \\
&&  + \; \int \left|\frac{\Phi_t}{\log|s|^{-1}} \right| \, d\mu_{t} \\ &<& 3 \eps.
\end{eqnarray*}

Finally, by the weak convergence of $\mu_s \to \hat{\mu}_f$ and convergence of $\phi(s,z)$ to $\hat{\phi}_f$, we can shrink $\delta$ again such that 
	$$\left|  \int\phi(s,z) \, d\mu_s - \int \hat{\phi}_f \, d\hat{\mu}_f \right|  < \eps$$
for all $|s| < \delta$.  Recalling the formula for $\hat{\phi}_f$ from \eqref{phi hat}, we have
	$$\int \hat{\phi}_f \, d\hat{\mu}_f = \int_0^1 \left( -a + \frac{a^2}{2} \right) \, da = -\frac{1}{3},$$
since the measure $\hat{\mu}_f$ is the uniform distribution on the interval $[0,1]$ in the $a$ coordinates, as described in \S\ref{Berkovich limit}.  Therefore,
$$\left( \frac{1}{6} - 4\eps \right)  \log |s|^{-1} \leq E_\infty(s,t) \leq \left( \frac{1}{6}  + 4\eps \right)  \log |s|^{-1}$$
for all $|s| < \delta$ and all $t\in C$.  
\end{proof}

\subsection{Both parameters escaping to the same cusp}

\begin{theorem}  \label{all b}
Given $\eps>0$, there exists $\delta>0$ such that
$$\left( \frac{1}{6} \left(1  - \frac{1}{b}\right)^2 - \eps \right)  \log |s|^{-1} \leq E_\infty(s,t) \leq \left( \frac{1}{6} \left(1  - \frac{1}{b}\right)^2 + \eps \right)  \log |s|^{-1}$$
for all $s,t$ satisfying $0 <  |s| \leq |t| < \delta$, where $b = (\log|s|)/(\log|t|) \geq 1$. 
\end{theorem}

For each real number $b\geq 1$, consider the function
\begin{eqnarray*}
\hat{\phi}_b(x) &:=&   \left\{ \begin{array}{ll}
				0 & \mbox{for }  |x|_0 \geq 1 \\
				\log|x|_0 - \dfrac{(\log |x|_0)^2}{2b\log|T|_0}  &  \mbox{for } |T|_0^b \leq  |x|_0 \leq 1  \\
				\dfrac{b}{2} \log |T|_0 & \mbox{for } |x|_0 \leq  |T|_0^b
				\end{array} \right.  \\
			&=& \left\{ \begin{array}{ll}
				0 & \mbox{for }  |x|_0  = |T|_0^a \mbox{ with } a \leq 0  \\
				-a + a^2/(2b)  &  \mbox{for } |x|_0 = |T|_0^a \mbox{ with }  0 \leq a \leq b  \\
				-b/2  & \mbox{for } |x|_0  = |T|_0^a \mbox{ with } a \geq b
				\end{array} \right.  
\end{eqnarray*}
for all $x \in \mathbb{L}$.  Note that $\hat{\phi}_1 = \hat{\phi}_f$ from \eqref{phi hat}.  As with $\hat{\phi}_1$, each $\hat{\phi}_b$ extends naturally to a function on the Berkovich projective line $\P^{1, an}_{k_0}$ and is a potential of the measure $\hat{\mu}_b - \delta_G$, where $\hat{\mu}_b$ is interval measure on $[\zeta_{0,|T|_0^b}, \zeta_{0, 1}]$ and $\delta_G$ is the delta-mass at the Gauss point $\zeta_{0,1}$.

For each $b \geq 1,$ the non-archimedean local energy $E_0(\hat{\mu}_1, \hat{\mu}_b)$ is given by 
\begin{eqnarray*}
E_0(\hat{\mu}_1, \hat{\mu}_b) 
&:=& \frac12 \left( \int (\hat{\phi}_1 - \hat{\phi}_b) \, d\hat{\mu}_b + \int (\hat{\phi}_b - \hat{\phi}_1) \, d\hat{\mu}_1 \right) \\ 
&=& \frac{\left( \log (|T|_0/|T|_0^b) \right)^2}{-6 \log \min \{|T|_0, |T|_0^b\} } \\
&=& \frac{(b-1)^2}{6b},
\end{eqnarray*}
as computed in Theorem \ref{non-arch energy} (in the case $v \nmid 2$).  

For $s$ and $t$ in $\C\setminus\{0,1\}$, if both $s$ and $t$ are close to one of the three cusps, we can estimate the archimedean local energy $E_\infty(s,t)$ in terms of the non-archimedean pairing using the degeneration description in hybrid space.  We first prove a special case of Theorem \ref{all b}:

\begin{prop}  \label{bounded b}
Given $\eps>0$ and $B>2$, there exists $\delta>0$ such that
	$$(E_0(\hat{\mu}_1, \hat{\mu}_b) - \eps) \log|t|^{-1} \leq E_\infty(s,t) \leq (E_0(\hat{\mu}_1, \hat{\mu}_b) + \eps) \log|t|^{-1}$$
for all $s, t$ satisfying $0 <  |t|^B \leq |s|  \leq |t| < \delta$, where $b = (\log|s|)/(\log|t|) \leq B$.   
\end{prop}

This proposition is an immediate consequence of the weak convergence of measures $\mu_t \to \hat{\mu}_1$ in the hybrid space, and the convergence of potentials as described in \S\ref{t near 0}.  We give the details to clarify how the bound $b\leq B$ is used.

\begin{proof}
Fix $\eps>0$ and $B>2$.

For $s$ and $t$ in the punctured unit disk $\D^*$, and for any $1 \leq b \leq B$, consider
\begin{equation} \label{phi t and s}
\phi(t,z) := \frac{\Phi_{F_t}(z,1)}{\log|t|^{-1}} \quad\mbox{ and } \quad b\phi(s,z) = \frac{b\Phi_{F_s}(z,1)}{\log|s|^{-1}} = \frac{\Phi_{F_s}(z,1)}{\log|t|^{-1}},
\end{equation}
viewed as functions on the fiber $\{t\} \times \C$ in the hybrid space. By Proposition \ref{uniform continuity}, there exists $\delta_1>0$ such that
\begin{equation} \label{g_t control}
	 \left|  \phi(t,z) - \hat{\phi}_1(\zeta_{0, |T|_0^a}) \right|  < \eps/(4B)
\end{equation}
for all $|t| < \delta_1$, $a\geq 0$, and all $|t|^{a + \delta_1}  < |z| <  |t|^{a - \delta_1}$.  In particular,
	$$|\phi(t,z)| < \eps/(4B)$$
for all $|z| \geq 1$ and all $|t| < \delta_1$.   It follows that 
	$$b\phi(s,z) \to \hat{\phi}_b(\zeta_{|T|_0^a})$$ 
in the hybrid space as $s$ and $t$ tend to 0 with $|s| = |t|^b$ and $(\log|z|)/(\log|t|) \to a$, uniformly in $b$ for $1 \leq b \leq B$.  This is because the annulus 
	$$A_t(a,\delta) := \{z \in \C:  |t|^{a + \delta}  < |z| <  |t|^{a - \delta}\}$$ 
for each fixed $t\in \D^*$, $\delta>0$ and $a\in\R$, can be written in terms of $s$ as
	$$A_t(a,\delta) = \left\{z \in \C:  \quad |s|^{a/b + \delta/b}  < |z| <  |s|^{a/b - \delta/b} \right\}$$
whenever $|s| = |t|^b$.  Therefore,
\begin{equation} \label{g_s control}
\left|  b\phi(s,z) - \hat{\phi}_b(\zeta_{0, |T|_0^a}) \right| 
	= \left|  \frac{b \, \Phi_{F_{s}}(z, 1)}{\log |s|^{-1}} - b \,\hat{\phi}_1(\zeta_{0, |T|_0^{a/b}}) \right| < b \, \eps/(4B) < \eps/4
\end{equation}
for all $z\in A_t(a,\delta_1)$, as a consequence of \eqref{g_t control}.  In particular,
	$$|b \,\phi(s,z)| < \eps/4$$
for all $|z| \geq 1$ and all $|t| < \delta_1$.  

%Thus the rate of convergence of $\phi_t^s$ to the functions $\hat{\phi}_b$ on the central fiber in $X^{hyb}$, as $t\to 0$, is uniform over all $1 \leq b \leq B$.

%We also need a uniform bound on the difference $\phi_t - \phi_t^s$ over the sets $\{|z| \geq 1\}$ and $\{|z| \leq |t|\}$ for all $t$ small.  Indeed, we already have that the two functions $\phi_t$ and $\phi_t^s$ are close to their limits in these regions.  As the limit functions are bounded in absolute value by $1/2$ and $b/2$, respectively, this implies a uniform bound of 
%\begin{equation} \label{disk around 0}	
%	|\phi_t - \phi_t^s| \leq \frac12 + \frac{b}{2} + \frac{\eps}{2}  \leq B
%\end{equation}
%on the disk $\{|z| \leq |t|\} \cup \{|z| \geq 1\}$ for all $|t| < \delta_1$.
%, and
%\begin{equation} \label{disk around infinity}
%	|\phi_t - \phi_t^s|  < \eps/2
%\end{equation}
%on the disk $\{|z| \geq 1\}$ for all $|t| < \delta_1$.  
%

Recall that the measures $\mu_t$ on the fiber over $t$ converge weakly in $X^{hyb}$ to the measure $\hat{\mu}_1$ on the central fiber.  For each $s$ with $|s| = |t|^b$, let $\mu_t^s$ denote the measure associated to $f_s$ but viewed in the fiber $\{t\}\times \P^1$.  The measures $\mu_t^s$ converge to the measure $\hat{\mu}_b$ as $t\to 0$ with $|s| = |t|^b$, and this convergence can also be made uniform in $b$ with $b\leq B$.  That is, by Lemma \ref{measure decomposition}, for any $n$ there exists $\delta_2>0$ such that 
\begin{equation} \label{mu_t control}
\frac{1}{n} - \frac{1}{n^2} < \mu_t( \{ |t|^{(i+1)/n} \leq |z| \leq |t|^{i/n} \}) < \frac{1}{n} + \frac{1}{n^2}
\end{equation}
for all $|t| < \delta_2$ and each $i = 0, \ldots, n$.  Note that this implies that 
	$$\mu_t(\{|z| \leq |t|\} \cup \{ |z| \geq 1\}) < \frac{1}{n}.$$
Therefore, we also have 
\begin{equation} \label{mu_s control}
\frac{1}{n} - \frac{1}{n^2} < \mu_t^s( \{ |t|^{b(i+1)/n} \leq |z| \leq |t|^{bi/n} \}) < \frac{1}{n} + \frac{1}{n^2}
\end{equation}
and
	$$\mu_t^s(\{|z| \leq |t|^b\} \cup \{ |z| \geq 1\}) < \frac{1}{n}.$$
for all $|t| < \delta_2$.  Thus, the measure $\hat{\mu}_b$ on small sub-annuli of the annulus $\{|t|^b \leq |z| \leq 1\}$ is controlled uniformly for all $1 \leq b \leq B$.

Putting all the pieces together,
$$\frac{E_\infty(s, t) }{\log|t|^{-1}} =\frac12 \left( \int \left( b\phi(s,z) - \phi(t,z) \right) \, d\mu_{t} + \int \left( \phi(t,z)- b\phi(s,z) \right) \, d\mu_t^s \right)$$
is within $\eps$ of 
	$$E_0(\hat\mu_1, \hat\mu_b) =  \frac12 \left( \int \left( \hat{\phi}_b - \hat{\phi}_1 \right) \, d\hat\mu_1 + \int \left( \hat{\phi}_1 - \hat{\phi}_b \right) \, d\hat\mu_b \right)   = \frac{(b-1)^2}{6b}$$
for all $t$ sufficiently small and all $s$ with $|s| = |t|^b$, for any $1 \leq b \leq B$.  
\end{proof}

Here is an equivalent restatement of Theorem \ref{all b}, expressed in terms of the growth of $|t|$:  

\begin{theorem}  \label{all b in t}
Given $\eps>0$, there exists $\delta>0$ such that
	$$\left(\frac{(b-1)^2}{6b} - b\, \eps\right) \log|t|^{-1} \leq E_\infty(s,t) \leq \left(\frac{(b-1)^2}{6b} + b\, \eps\right) \log|t|^{-1} $$
for all $s, t$ satisfying $0 <  |s| \leq |t| < \delta$, where $b = (\log |s|)/(\log|t|) \geq 1$. 
\end{theorem}

Comparing Theorem \ref{all b in t} to the statement of Proposition \ref{bounded b}, we see that we lose the ability to bound the energy within a uniform $\eps$ when $b$ becomes large.  

\begin{proof}[Proof of Theorems \ref{all b in t} and \ref{all b}]
Fix $\eps>0$.  

As in the proof of Proposition \ref{bounded b}, we make use of the weak convergence of measures $\mu_t \to \hat{\mu}_1$ and convergence of the potentials $\phi(t,z) \to \hat{\phi}_1$ in the hybrid space as $t \to 0$.  Recalling the formula for $\hat{\phi}_f$ from \eqref{phi hat}, we have
	$$\int \hat{\phi}_f \, d\hat{\mu}_f = \int_0^1 \left( -a + \frac{a^2}{2} \right) \, da = -\frac{1}{3},$$
since the measure $\hat{\mu}_f$ is the uniform distribution on the interval $[0,1]$ in the $a$ coordinates, as described in \S\ref{Berkovich limit}.   

Choose $r$ satisfying $0 < r < \eps/100$.  There is a $\delta_2$ such that 
	$$\mu_t(\{|z| \leq |t|\}) < \eps/50$$ 
and 
	$$|\phi(t,z)| < \eps/50 \mbox{ for } |z| \geq |t|^r$$ 
and 
	$$|\phi(t,z)| < \frac12 + \eps/50 \mbox{ for all } z$$
for all $|t| < \delta_2$.  Thus, for $s \in \C^*$ with $|s| = |t|^b$ and $b > 1/r$,  we have
	$$|b \phi(s,z)| < b \,\eps/50 \mbox{ for } |z| \geq |t|$$
and 
	$$|b \phi(s,z)| < \frac{b}{2} + b\,\eps/50 \mbox{ for all } z$$
for all $|t| < \delta_2$. By shrinking $\delta_2$ further if necessary, we appeal to the weak convergence of measures $\mu_t \to \hat{\mu}_1$ in the hybrid space to deduce that 
\begin{equation} \label{integral limit}
\left| \int \phi(t,z) \, d\mu_t + \frac{1}{3} \right| = \left|  \int \phi(t,z) \, d\mu_t - \int \hat{\phi}_1 \, d\hat{\mu}_1 \right|  \; < \; \eps/10
\end{equation}
for all $|t| < \delta_2$.  

Now fix $B>1/r$, and recall that $r < \eps/100$, so that 
\begin{equation} \label{choice of B}
	\left| \frac{(b-1)^2}{6 \, b} - \frac{b}{6} \right| <  b \, \eps/50
\end{equation}
for all $b \geq B$.  For this $B$, we can find a $\delta_B >0$ such that Proposition \ref{bounded b} is satisfied for all $0 < |s| = |t|^b \leq |t| < \delta_B$ with $1 \leq b \leq B$.  Choose any $\delta \leq \min\{\delta_B, \delta_2\}$, and we obtain the theorem for $b \leq B$.

Now suppose $b\geq B$.  We will estimate 
$$\frac{E_\infty(s, t) }{\log|t|^{-1}} = \frac12 \left( \int \left( b\phi(s,z) - \phi(t,z) \right) \, d\mu_{t} + \int \left( \phi(t,z) - b \phi(s,z) \right) \, d\mu_t^s \right)$$ 
for all $|t| < \delta$ and any $s$ with $|s| = |t|^b$ by estimating the two integrals separately.

As shown above,
	$$|b\phi(s,z) - \phi(t,z)| \leq 1/2 + \eps/50 + b\,\eps/50 < b \, \eps/10$$
for all $|z| \geq |t|$ and $0 < |s| = |t|^b \leq |t| < \delta$ with $b\geq B$, and 
	$$|b\phi(s,z) - \phi(t,z)| \leq 1/2 + \eps/50 + b/2 +  b \,\eps/50$$
for all $z$ and $0 < |s| = |t|^b \leq |t| < \delta$.  
Writing the first integral as 
 $$\int \left(b\phi(s,z) - \phi(t,z) \right) \, d\mu_{t} = \int_{|z|\geq |t|} (b\phi(s,z) - \phi(t,z)) \, d\mu_t + \int_{|z|\leq |t|} (b\phi(s,z) - \phi(t,z)) \, d\mu_t,$$
it follows that 
$$\left| \int (b\phi(s,z) - \phi(t,z)) \, d\mu_t \right| \leq b\, \eps/10 + (1/2 + \eps/50 + b/2 +  b \,\eps/50) (\eps/50)  < b \,\eps/ 5$$ 
for all $b\geq B$ and $0 < |s| = |t|^b \leq |t| < \delta$.

Write the second integral as 
$$\int \left( \phi(t,z) - b\phi(s,z) \right) \, d\mu_t^s = \int \phi(t,z)  \, d\mu_t^s - \int b\phi(s,z) \, d\mu_t^s.$$
As $|\phi(t,z)|$ is bounded by $1/2 + \eps/50$, we have
	$$\left| \int \phi(t,z) \, d\mu_t^s \right| \leq \frac12 + \eps/50 <  b \, \eps/25$$
for all $b\geq B$.   On the other hand, we have 
	$$\int b \phi(s,z) \, d\mu_t^s = b  \int\phi(s,z) \, d\mu_s$$
so that 
$$ \left| \int b \phi(s,z) \, d\mu_t^s + \frac{b}{3} \right| < b \, \eps/10$$
for all $0 < |s| = |t|^b \leq |t| < \delta$ from \eqref{integral limit}.  

We conclude that 
	$$\left|\frac{E_\infty(s, t) }{\log|t|^{-1}} - \frac{b}{6}\right| <  b\,\eps/2$$
for all $b$ sufficiently large and all $|t| < \delta$.  On the other hand, we also have 
	$$\left| \frac{b}{6} - \frac{(b-1)^2}{6b}\right| < b\, \eps/50$$
for all $b\geq B$ by our choice of $B$, so the theorem is proved.
\end{proof}

\subsection{Parameters escaping to different cusps}

\begin{theorem} \label{different cusps}
Given $\eps>0$, there exists $\delta>0$ such that 
$$\left( \frac{1}{6} \left(1  + \frac{1}{b}\right) - \eps \right)  \log |s|^{-1} \leq E_\infty(s,t) \leq \left( \frac{1}{6} \left(1  + \frac{1}{b}\right) + \eps \right)  \log |s|^{-1}$$
for all $s,t \in \C$ satisfying $|t| > 1/\delta$ and $0 < |s| \leq 1/|t|$, where $b = -(\log|s|)/(\log|t|)$. 
\end{theorem}

\begin{proof}
The proof is nearly identical to that of Theorem \ref{all b}, working in the hybrid space over a unit disk that we will parameterize by $u \in \D$.  For fixed $b\geq 1$, $t = 1/u$ and any $s$ satisfying $|s| = |u|^b$, consider the functions
	$$g_{u}(z)  = \frac{G_{F_{1/u}}(z,1)}{\log|u|^{-1}} = \frac{G_{F_{t}}(z,1)}{\log|t|}$$
and 
	$$b g_s(z) =  b\frac{G_{F_{s}}(z,1)}{\log|s|^{-1}} $$
in the fiber $\{u\} \times \P^1$.  

As computed in Proposition \ref{big t escape-rate formula}, the limit of $g_u(z)$ as $u\to 0$ with $|z| = |u|^a$ is 
	$$\hat{g}_\infty(a) := \left\{  \begin{array}{ll}
				-a  & \mbox{for } a \leq -1 \\
				(a^2 + 1)/2 & \mbox{for } -1 \leq a \leq 0 \\
				1/2 & \mbox{for } a\geq 0  \end{array} \right.$$
As $u \to 0$, the measures $\mu_{1/u}$ on $\{u\}\times \P^1$ will to converge the canonical measure $\hat{\mu}_\infty$ for the map $f = f_{1/U}$ on the Berkovich projective line, working over the field $\C((U))$; the measure $\hat{\mu}_\infty$ is uniformly distributed on the interval $[\zeta_{0,1}, \zeta_{0,|U|_0^{-1}}]$.

As $s  \to 0$ with $|s| = 1/|t|^b = |u|^b$, $b\geq 1$, we have 
	$$b g_s(z) \rightarrow \hat{g}_b(\zeta_{0, |U|_0^a})$$
for $|z| = |u|^a$, exactly as in \eqref{g_s control}.  The non-archimedean local energy is computed in Theorem \ref{non-arch energy} as 
	$$E(\hat{\mu}_\infty, \hat{\mu}_b) = \frac{b+1}{6}.$$
We conclude as in the proof of Theorem \ref{all b} that, for all given $\eps>0$, there exists $\delta>0$ such that
$$\left( \frac{b+1}{6}  - b\,\eps \right)  \log |t| \leq E_\infty(s,t) \leq \left( \frac{b+1}{6}  + b\,\eps \right)  \log |t|$$
for all $|t| > 1/\delta$ and $|s| = 1/|t|^b$.  This completes the proof of Theorem \ref{different cusps}.
\end{proof}

%%%%%%
\bigskip
\section{Proof of Theorems \ref{uniformlowerbound} and \ref{pairingasymptotics}} \label{pairingsection}

In this section, we first prove Theorem \ref{pairingasymptotics}, which states there exist constants $\alpha, \beta>0$ such that 
	$$\hat{h}_{t_1} \cdot \hat{h}_{t_2} \; \geq \; \alpha \, h(t_1, t_2) - \beta$$
for all $t_1 \not= t_2$ in $\overline{\mathbb{Q}} \setminus \{0,1\}$. We then use this lower bound to prove Theorem \ref{uniformlowerbound} and Proposition \ref{uniformZhang}.

\subsection{Balancing local contributions}
Fix any $r$ such that 
  $$0<r\leq 1/16.$$
Fix $t_1, t_2 \in \Qbar\setminus\{0,1\}$, and let $K$ be any number field containing $t_1$ and $t_2$.  
We split the places $M_K$ into ``good" and ``bad" subsets, depending on the pair $t_1, t_2$ and the choice of $r$.  Let $M_\good(t_1, t_2)$ be the set of places $v\in M_K$ with   
 $$|\log|t_2/t_1|_v|\geq r\cdot\max\{ |\log|t_2|_v|, |\log|t_1|_v|\};$$
and set $M_\bad(t_1, t_2) =M_K\backslash M_\good(t_1, t_2)$.  We further decompose $M_\good(t_1,t_2)$ into its archimedean ($M_\good^\infty$) and non-archimedean ($M_\good^0$) places.
    
\begin{lemma}\label{arch-estimate lemma}
There exists a constant $C_0 >0$ such that 
   $$6\,E_v(t_1,t_2)\geq \frac{3r}{4}|\log |t_1/t_2|_v| - C_0$$
for any choice of $t_1$ and $t_2$ in $\Qbar\setminus\{0,1\}$ and for all $v\in M_\good^{\infty}(t_1, t_2)$.
\end{lemma}

\begin{proof}
Let 
  $$\epsilon=\frac{r^2}{24} $$
and let $\delta_1$ be the minimum of the $\delta$'s from Theorems \ref{all b} and \ref{different cusps} for this choice of $\epsilon$. Let $\delta_2$ be the $\delta$ of Theorem \ref{s to 0} for the compact set 
   $$\{t\in \C:  \delta_1\leq |t|\leq 1/\delta_1\textup{ and } |t-1|\geq \delta_1\}$$
in $\C\setminus\{0,1\}$.  Let $\delta_0$ be the minimum of $\delta_1$ and $\delta_2$, and let $C_0$ be any real number larger than $\log(1/\delta_0)$.  

Now fix $t_1, t_2$ and any number field $K$ containing $t_1$ and $t_2$, and fix a place $v \in M_\good^{\infty}(t_1, t_2)\subset M_K$.  If $\delta_0\leq |t_i|_v\leq 1/\delta_0$ for $i=1,2$, we have 
   $$6\, E_v(t_1, t_2)\; \geq \; 0 \;\geq \; \frac{3r}{4} |\log |t_1/t_2|_v|- C_0.$$
As $v \in M_\good^\infty$, if $|t_2|_v\leq |t_1|_v<1$, then 
	$$|t_2|_v=|t_1|_v^b<|t_1|_v \mbox{ for } \frac{b-1}{b} \geq r$$
and therefore, by Theorem \ref{all b}, if additionally $|t_1|_v < \delta_1$, then we have
\begin{eqnarray*}
6\, E_v(t_1, t_2)&\geq& \left(\frac{(b-1)^2}{b^2} -6\epsilon\right)|\log |t_1|_v^b|\\
&\geq& \left(\frac{(b-1)r}{b} -\frac{r^2}{4}\right)|\log |t_1|_v^b|\\
&\geq& \frac{3r}{4}\frac{b-1}{b}|\log |t_1|_v^b|=\frac{3r}{4}|\log|t_1/t_2|_v|.
\end{eqnarray*}
If $|t_1|_v < \delta_1$ and $|t_2|_v=1/|t_1|_v^b$ for some $b\geq 1$, then by Theorem \ref{different cusps}
   $$6\, E_v(t_1, t_2)\geq \left(\frac{(b+1)}{b} -6\epsilon\right)|\log |t_1|_v^b|\geq \frac{3(b+1)r}{4}|\log |t_1|_v|=\frac{3r}{4}|\log|t_1/t_2|_v|;$$
if $\delta_1\leq |t_1|_v\leq 1/\delta_1$, $|t_1-1|_v\geq \delta_1$ and $|t_2|_v< \delta_0$, we have by Theorem \ref{s to 0} that
   $$6\, E_v(t_1, t_2)\geq (1-6\epsilon)|\log |t_2|_v|\geq \frac{3r}{4}|\log|t_1/t_2|_v|- C_0.$$
Combining the above inequalities with the symmetry relations of Proposition \ref{one sing}, we obtain 
  $$6\, E_v(t_1, t_2) \geq \frac{3r}{4}|\log|t_1/t_2|_v|- C_0.$$
\end{proof}

\begin{lemma}\label{half bound}
There is a constant $C>0$ such that 
   $$\sum_{v\in M_\good(t_1, t_2)} 3r_vE_v(t_1, t_2)\geq \frac{3r}{4}h(t_2/t_1)-\frac{3r^2}{2}h(t_1,t_2)- C$$
for any $t_1\neq t_2\in \Qbar\backslash\{0, 1\}$.
\end{lemma}

\begin{proof}
Fix $t_1$ and $t_2$ and any number field $K$ containing them.  For the non-archimedean places $v\in M_\good^0(t_1, t_2)$, by Theorem \ref{non-arch energy}, we have 
   $$6E_v(t_1, t_2)\geq r\cdot |\log|t_2/t_1|_v|-8\log^+|1/2|_v,$$
and thus
\begin{equation}\label{non-archi lower bound 11}
\sum_{v\in M_\good^0(t_1, t_2)}6r_vE_v(t_1, t_2) \geq \sum_{v\in M_\good^0(t_1, t_2)}r_v\,\left(r\cdot |\log|t_2/t_1|_v|-8\log^+|1/2|_v\right).
\end{equation}
Now choose any integer $N_0$ so that $\log N_0$ is larger than the $C_0$ of Lemma \ref{arch-estimate lemma}, for each archimedean $v\in M_\good^\infty(t_1, t_2)$.  We have
   $$6E_v(t_1, t_2)\geq \frac{3r}{4}|\log |t_2/t_1|_v|-\log N_0$$
for all $v\in M_\good^\infty(t_1, t_2)$.  With $h$ the naive logarithmic height on $\Qbar$, we set
	$$C = 4 \, h(2) + \frac12 \, h(N_0) = \frac12 \log (2^8 N_0).$$
Then, we have
\[\begin{split}
\sum_{v\in M_\good(t_1, t_2)}&6r_vE_v(t_1, t_2)=\sum_{v\in M_\good^\infty(t_1, t_2)}6r_vE_v(t_1, t_2)+\sum_{v\in M_\good^0(t_1, t_2)}6r_vE_v(t_1, t_2)\\
&\geq \sum_{v\in M_\good(t_1, t_2)}r_v\, \frac{3r}{4}|\log|t_2/t_1|_v| -\sum_{v\in M_K}r_v\,\left(8\log^+|2|_v+\log^+|N_0|_v\right)\\
&= \sum_{v\in M_\good(t_1, t_2)}r_v\,\frac{3r}{4}|\log|t_2/t_1|_v| -2C\\
&=   \sum_{v\in M_K}r_v\, \frac{3r}{4}|\log|t_2/t_1|_v|-\sum_{v\in M_\bad(t_1, t_2)}r_v\, \frac{3r}{4}|\log|t_2/t_1|_v| -2C\\
&\geq  \sum_{v\in M_K}\frac{3r_vr}{4}\, |\log|t_2/t_1|_v|-\sum_{v\in M_\bad(t_1, t_2)}\frac{3r^2r_v}{4} \max\{ |\log|t_2|_v|, |\log|t_1|_v|\} -2C\\
&\geq  \frac{3r}{4}\, 2 \, h(t_2/t_1)-\frac{3r^2}{4}\cdot 4\, h(t_2, t_1) -2C
\end{split}
\]
For the last inequality, we use the facts that $2\, h(x)=\sum_{v\in M_K}r_v|\log|x|_v|$ for nonzero $x \in K$ and 
  $$\sum_{v\in M_K}r_v\,\max\{ |\log|t_2|_v|, |\log|t_1|_v|\}\leq2\left(h(t_2)+h(t_1)\right)\leq 4\, h(t_2,t_1).$$
\end{proof}

%%%
%%%
\subsection{Proof of Theorem \ref{pairingasymptotics}} 
We begin with a standard lemma.

\begin{lemma}\label{height bound}
There is a constant $C>0$, such that 
  $$h\left(\frac{t_2}{t_1}\, , \, \frac{1-t_2}{1-t_1}\right)\geq \frac{1}{2}\, h(t_1, t_2)-C$$
for $t_1\neq t_2\in \Qbar\setminus\{0, 1\}$. Here the $h$ is the naive logarithmic height on $\A^2(\Qbar)$. 
\end{lemma} 

\begin{proof} 
Consider the birational transformation $g: \P^2 \dashrightarrow \P^2$ defined in affine coordinates by $g(x_1, x_2)=(x_2/x_1, (1-x_2)/(1-x_1))$, with inverse 
	$$g^{-1}(y_1, y_2)=\left(\frac{1-y_2}{y_1-y_2}, \frac{y_1(1-y_2)}{y_1-y_2}\right)$$
of degree $d=2$.  There exists a constant $C$ such that 
$$h(g^{-1}(x:y:z))\leq (\deg g^{-1}) \, h(x:y:z)+C = 2 \, h(x:y:z)+C$$
outside of the indeterminacy set for $g^{-1}$ in $\P^2$ \cite[Theorem B.2.5]{Hindry:Silverman}.  The indeterminacy set for $g^{-1}$ is $\{(0:1:0), (1:0:0), (1:1:1)\}$.  Therefore, letting $(t_1:t_2:1) = g^{-1}(x:y:1)$ for some point $(x:y:1)$, we obtain
  $$h(t_1, t_2)\leq 2\, h(g(t_1, t_2))+C$$
for all $t_1 \not= t_2$ in $\Qbar \setminus \{0,1\}$.  In other words,
  $$h\left(\frac{t_2}{t_1}, \frac{t_2-1}{t_1-1}\right)\geq \frac{1}{2}\, h(t_1, t_2)-\frac12 \, C.$$
\end{proof}

Now fix $t_1 \not= t_2$ in $\Qbar\setminus\{0,1\}$.   From Lemma \ref{half bound}, we know there is a constant $C$ (independent of $t_1$ and $t_2$) such that
\begin{eqnarray}\label{one lower bound 1}
3\,\hat{h}_{t_1} \cdot \hat{h}_{t_2} &=& 3 \sum_{v \in M_K} r_v\, E_v(t_1, t_2) \nonumber \\
 &\geq& 3 \sum_{v \in M_\good(t_1, t_2)} r_v\, E_v(t_1, t_2) \nonumber \\
 &\geq&    \frac{3r}{4}h(t_2/t_1)-\frac{3r^2}{2}h(t_2,t_1)- C
\end{eqnarray}
for any $t_1\neq t_2\in \Qbar\backslash\{0,1\}$.  Replacing $t_i$ in inequality \eqref{one lower bound 1} with $1-t_i$, for $i=1, 2$, we also have
      $$3\,\hat{h}_{1-t_1} \cdot \hat{h}_{1-t_2}\geq \frac{3r}{4}h\left(\frac{1-t_2}{1-t_1}\right)-\frac{3r^2}{2}h(1-t_2,1-t_1)- C.$$
Combining this with Proposition \ref{one sing}, we find that
\begin{equation}\label{one lower bound 2}
3\, \hat{h}_{t_1} \cdot \hat{h}_{t_2}\geq \frac{3r}{4}h\left(\frac{1-t_2}{1-t_1}\right)-\frac{3r^2}{2}h(1-t_2,1-t_1)-C.
\end{equation}
Consequently, by adding the inequalities \eqref{one lower bound 1} and \eqref{one lower bound 2}, we have
      $$6\,\hat{h}_{t_1} \cdot \hat{h}_{t_2}\geq \frac{3r}{4}\left(h(t_2/t_1)+h\left(\frac{1-t_2}{1-t_1}\right)\right)-\frac{3r^2}{2}\left(h(t_2,t_1)+h(1-t_2,1-t_1)\right)-2\, C.$$
Observe that there is a constant $C'>0$ from Lemma \ref{height bound} such that 
         $$h(t_2/t_1)+h\left(\frac{1-t_2}{1-t_1}\right)\geq h\left(\frac{t_2}{t_1}, \frac{1-t_2}{1-t_1}\right) \geq \frac{1}{2}\cdot h(t_1,t_2)-C'.$$
Since $|h(1-t_2, 1-t_1)-h(t_1, t_2)|$ is uniformly bounded over all pairs $t_1, t_2 \in \Qbar$, we may combine the above inequality with the previous to conclude that 
     $$6\, \hat{h}_{t_1} \cdot \hat{h}_{t_2}\geq \frac{3r}{8} h(t_1, t_2)-{3r^2}\cdot h(t_1,t_2)- 6\,C'',$$
In other words, 
  $$\hat{h}_{t_1} \cdot \hat{h}_{t_2}\geq (r/16-r^2/2)h(t_2,t_1)-C'',$$
and the proof of Theorem \ref{pairingasymptotics} is complete by taking $\alpha=r/16-r^2/2$ and $\beta=C''$. \qed

\begin{remark}
If we set $r=1/16$, the constant $\alpha>0$ in Theorem \ref{pairingasymptotics} can be taken to be $\alpha=1/512$.
\end{remark}

\subsection{Proof of Theorem \ref{uniformlowerbound}} 
We will use Theorem \ref{pairingasymptotics} to deduce a uniform lower bound on the height pairing $\hat{h}_s\cdot \hat{h}_t$ for all $s \ne t$ in $\Qbar\setminus\{0, 1\}$.

Suppose there exist parameters $s_n \not= t_n \in \overline{\mathbb{Q}}$ such that 
$$\hat{h}_{s_n} \cdot \hat{h}_{t_n} \rightarrow 0 \ \ \ \ \text{ as } n \rightarrow \infty.$$
Fix $\epsilon > 0$.  For each $n$, choose a number field $K_n$ containing $s_n$ and $t_n$.  By assumption and non-negativity of the local energies $E_v$ (Proposition \ref{local energy}), there is $N \in \mathbb{N}$ such that for all $n > N$, the archimedean contribution to the pairing is less than $\epsilon$; that is, for $n > N$, 
$$\sum_{v \in M_{K_n}^{\infty}} r_v \, E_v({s_n}, {t_n}) < \epsilon,$$
recalling that $r_v = \frac{[{K_n}_v: \mathbb{Q}_v]}{[K_n:\mathbb{Q}]}$ now depends on $n$.  

Let $M_n$ be the set of archimedean places $v$ in $M_{K_n}^{\infty}$ such that $E_v({s_n}, {t_n}) < 2 \epsilon$, noting that for $n > N$, we have
$$\sum_{v \in M_n} r_v \geq \frac{1}{2}.$$
Recall that the local energy $E_v(s,t)$ is continuous in $s$ and $t$, and it vanishes if and only if $s=t$.  So 
there exists a $\delta$, depending only on $\epsilon$, so that, for each $n>N$ and for each place $v \in M_n$, one of the following must hold:
\begin{enumerate}
\item $|t_n - s_n|_v < \delta$ 
\item $\min \{ |s_n|_v, |t_n|_v \} < \delta$ 
\item $\min \{ |s_n-1|_v, |t_n-1|_v \} < \delta$ 
\item $\max \{ |s_n|_v, |t_n|_{v} \} >  1/\delta$
\end{enumerate}
Note that we can take $\delta \rightarrow 0$ as $\epsilon \rightarrow 0$.  We may then, for each $n>N$, choose a subset $M_n'$ of $M_n$ for which $s_n$ and $t_n$ satisfy the same one of the four conditions at all places $v \in M_n'$, and such that 
$$\sum_{v \in M_n'} r_v \geq \frac{1}{8}.$$
We conclude by the product formula that 
$$\max \{ h(s_n - t_n), h(s_n, t_n), h(s_n - 1, t_n -1) \} > \frac{1}{8} \log \frac{1}{\delta}.$$
It then follows from the triangle inequality, combined with shrinking our choice of $\epsilon$, that we have $h(s_n, t_n) \rightarrow \infty$.  The inequality of Theorem \ref{pairingasymptotics} implies that $\hat{h}_{s_n} \cdot \hat{h}_{t_n} \rightarrow \infty$ as well, a contradiction.  This completes the proof of Theorem \ref{uniformlowerbound}.  

\subsection{Proof of Proposition \ref{uniformZhang}}
Fix a number field $K$, and fix $t_1\not=t_2$ in $K \setminus\{0,1\}$.  Let $\| \cdot \|_i$ denote the adelic metric on the line bundle $\mathcal{O}_{\P^1}(1)$ associated to the height $\hat{h}_{t_i}$.  Let $\overline{L}$ denote the line bundle $\mathcal{O}_{\P^1}(1)$ equipped with the metric $(\|\cdot \|_1 \| \cdot \|_2)^{1/2}$; its associated height function is
	$$h_{\overline{L}}(x) = \frac12 \left(\hat{h}_{t_1}(x) + \hat{h}_{t_2}(x)\right).$$
Zhang's inequality on the essential minimum of a height function implies that 
	$$\liminf_{n \to \infty} h_{\overline{L}}(x_n) \; \geq  \;  (h_{\overline{L}}\cdot h_{\overline{L}})/(2 \deg L) = \frac12 \left( h_{\overline{L}}\cdot h_{\overline{L}} \right)$$
along any infinite sequence of distinct points $x_n \in \P^1(\Qbar)$ \cite[Theorem 1.10]{Zhang:adelic}.  In particular, the set
	$$\{ x \in \P^1(\Qbar):  \hat{h}_{t_1}(x) + \hat{h}_{t_2}(x)  \leq b \}$$
is finite for any choice of $b < h_{\overline{L}}\cdot h_{\overline{L}}$.  

By the linearity of the intersection pairing, we see that 
	$$h_{\overline{L}}\cdot h_{\overline{L}} = \frac14 \, \hat{h}_{t_1}\cdot \hat{h}_{t_1} + \frac12 \, \hat{h}_{t_1}\cdot \hat{h}_{t_2} + \frac14 \,\hat{h}_{t_2}\cdot \hat{h}_{t_2} = \frac12 \, \hat{h}_{t_1}\cdot \hat{h}_{t_2}.$$
Therefore, we may choose any $b < \delta/2$ for the $\delta$ of Theorem \ref{uniformlowerbound}, and the proposition is proved.

%%%%%%
%%%%%%

\bigskip
\section{Proof of Theorem \ref{boundingN}} \label{trianglesection}

Fix any $b\geq 0$ such that $b < \delta/2$ for the $\delta$ of Theorem \ref{uniformlowerbound}.  Recall from Proposition \ref{uniformZhang} that the set
\begin{equation} \label{S def}
	S(b, t_1, t_2) := \{ x \in \P^1(\Qbar):  \hat{h}_{t_1}(x) + \hat{h}_{t_2}(x) \leq b\}
\end{equation}
is finite for every pair $t_1 \not= t_2 \in \Qbar\setminus\{0,1\}$. Note that $\{0,1,\infty\} \subset S(b,t_1,t_2)$ so that $|S(b,t_1,t_2)| \geq 3$  for all $t_1 \not= t_2$ in $\Qbar\setminus\{0,1\}$ and all $b\geq 0$.  In this section, we prove the following generalization of Theorem \ref{boundingN}.  

\begin{theorem} \label{boundingNb}
Let $b\geq 0$ be chosen so that $b < \delta/2$ for the $\delta$ of Theorem \ref{uniformlowerbound}. 
For all $\eps>0$, there exists a constant $C(\eps)$ so that 
  $$ \hat h_{t_1}\cdot \hat h_{t_2} \leq 4b + \left( \eps + \frac{C(\eps)}{|S(b,t_1,t_2)|} \right) (h(t_1, t_2) + 1),$$
for all $t_1\not= t_2$ in $\Qbar\setminus\{0,1\}$, for the set $S(b,t_1,t_2)$ defined by \eqref{S def}.
\end{theorem}

\noindent
Note that Theorem \ref{boundingN} follows from Theorem \ref{boundingNb} by setting $b=0$.

\subsection{Adelic measures and heights associated to a finite set}
Fix a number field $K$, and suppose that $F$ is a finite set in $\Kbar$ which is $\Gal(\Kbar/K)$-invariant.  Let $\eta$ be a collection of positive real numbers
   $$\eta:=\{\eta_v\}_{v\in M_K}$$
with $\eta_v=1$ for all but finitely many $v\in M_K$.  For archimedean $v\in M_K$ and $x \in F$, we let $m_{x,v}$ denote the Lebesgue probability measure on the circle of radius $\eta_v$ centered at the point $x \in F$.  We then set 
	$$m_{F, \eta, v} = \frac{1}{|F|}\sum_{x\in F}m_{x, v}.$$
Similarly, for each non-archimedean $v \in M_K$, we let $m_{F,\eta, v}$ denote the probability measure distributed uniformly on the points $\zeta_{x, \eta_v}$ in $\A_v^{1,an}$ over all $x \in F$.  Then $m_{F,\eta} = \{m_{F, \eta, v}\}_{v\in M_K}$ is an adelic measure in the sense of \cite{FRL:equidistribution}. It gives rise to a unique height $h_{F,\eta}$ on $\P^1(\Qbar)$ associated to a continuous and semipositive adelic metric on $\mathcal{O}_{\P^1}(1)$ with curvature distributions given by $m_{F, \eta, v}$ and satisfying 
\begin{equation} \label{eta normalized}
  h_{F,\eta}\cdot h_{F,\eta}=0.
\end{equation}
Its local heights are given by 
	$$\lambda_{F,\eta, v}(z) = \alpha_v + \frac{1}{|F|} \sum_{x \in F} \log\max\{|z - x|_v, \eta_v\},$$
for $z\in \C_v$ and suitable constants $\alpha_v$; taking 
	$$\alpha_v = -\frac{1}{2\, |F|} \sum_{x \in F} \int\log\max\{|z - x|_v, \eta_v\} \, d m_{F,\eta,v}$$ 
gives \eqref{eta normalized}.

\begin{remark}
The height $h_{F,\eta}$
% is comparable to the standard logarithmic Weil height $h$, meaning that there is a constant $C = C(F,\eta)$ so that $|h - h_{F_\eta}| \leq C$ on $\P^1(\Qbar)$, but it 
will generally not admit sequences of ``small" points, meaning sequences $x_n \in \P^1(\Qbar)$ with $h_{F,\eta}(x_n) \to 0$.  In fact, for any choices of $F$ and $\eta$ such that $\sum_v r_v \, \alpha_v \not= 0$, the essential minimum of $h_{F, \eta}$ is positive.  
\end{remark}

\subsection{An upper bound on the height pairing}
Now suppose that $t_1$ and $t_2$ lie in $K\setminus\{0,1\}$.  Recall that $\mu_t$ and $\hat{h}_t$ respectively denote the measure and height associated to the curve $E_t$.  By the triangle inequality for the distance function of \S\ref{metric definition}, we have 
\begin{equation} \label{triangle}
   \left(\hat h_{t_1}\cdot \hat h_{t_2}\right)^{1/2}\leq \left(\hat h_{t_1}\cdot h_{F,\eta}\right)^{1/2}+\left(\hat h_{t_2}\cdot h_{F,\eta}\right)^{1/2} 
\end{equation}
for any choice of $F$ and $\eta$.  By symmetry and bilinearity of the mutual energy,
\begin{eqnarray*}  
 \hat h_{t_i} \cdot h_{F,\eta} &=& \frac{1}{2} \sum_{v \in M_K} r_v\, (\mu_{t_i,v} - m_{F,\eta, v}, \mu_{t_i, v} - m_{F,\eta, v})_v \\
 	&=&   \frac{1}{2} \sum_{v \in M_K} r_v\, \left( (\mu_{t_i,v}, \mu_{t_i,v})_v - 2 \,( m_{F,\eta, v}, \mu_{t_i, v})_v  + (m_{F,\eta,v}, m_{F,\eta, v})_v\right)
\end{eqnarray*}
for $i = 1,2$.  For fixed $i$, writing the local height for $\hat{h}_{t_i}$ as $\lambda_{t_i, v} = \log|z|_v + c_v + o(1)$ for $|z|_v \to \infty$ yields
	$$\sum_v r_v \, (\mu_{t_i, v}, \mu_{t_i, v})_v =  - \sum_v r_v \int (\lambda_{t_i,v} - c_v) \, d\mu_{t_i,v} = 0$$
from \eqref{zero sum}.   Therefore,  
   $$\hat h_{t_i} \cdot h_{F,\eta} =  \frac{1}{2} \sum_{v \in M_K} r_v\, \left(-2(\mu_{t_i,v}, m_{F,\eta, v})_v + (m_{F,\eta,v}, m_{F,\eta,v})_v\right).$$
   
Recall from \S\ref{Mutualenergy} that $[F]_v$ is the probability measure on $\P_v^{1,an}$ distributed equally on the elements of $F$ for each $v\in M_K$.  By \cite[Lemma 4.11]{FRL:equidistribution} and \cite[Lemma 12]{Fili:energy}, we have 
	$$(m_{F,\eta,v}, m_{F,\eta,v})_v \leq ([F]_v, [F]_v)_v + \frac{-\log \eta_v}{|F|}.$$
It follows that 
\begin{eqnarray} \label{upper bound almost} 
\hat h_{t_i} \cdot h_{F,\eta} &\leq&  \frac{1}{2} \sum_{v \in M_K} r_v\cdot\left(-2\,(\mu_{t_i,v}, m_{F,\eta,v})_v + ([F]_v, [F]_v)_v + \frac{-\log \eta_v}{|F|}\right)  \nonumber \\
&=& \frac{1}{2} \sum_{v \in M_K} r_v\cdot \left( -2\,(\mu_{t_i,v}, m_{F,\eta, v})_v + \frac{-\log \eta_v}{|F|}\right) 
\end{eqnarray}
with the final equality following from \eqref{finite sets}.

\begin{prop}  \label{upper bound with F}
Suppose $t \not= 0,1$ lies in a number field $K$.  Assume that $F$ is a finite, $\Gal(\Kbar/K)$-invariant set of points.  Then 
$$\hat h_t \cdot h_{F,\eta} \; \leq \;   \hat{h}_t(F) + \sum_{v \in M_K} r_v\, \left( - (\mu_{t,v}, m_{F,\eta, v})_v+ (\mu_{t,v}, [F]_v)_v + \frac{-\log \eta_v}{2 \, |F|}\right).$$
for any choice of $\eta = \{\eta_v\}_v$ with $\eta_v = 1$ for all but finitely many $v \in M_K$.  
\end{prop}

\begin{proof}
The height of $F$ is computed as 
  $$\hat{h}_t(F) = \frac{1}{|F|}\sum_{x\in F} \hat h_t(x)=\hat{h}_t(\infty) -\sum_{v\in M_K} r_v\, (\mu_{t,v}, [F]_v)_v=-\sum_{v\in M_K} r_v\, (\mu_{t,v}, [F]_v)_v,$$
and therefore we may add $\hat{h}_t(F) + \sum_v r_v\, (\mu_{t, v}, [F]_v)_v$ to the right hand side of \eqref{upper bound almost}.  
\end{proof}

\subsection{Proof of Theorem \ref{boundingNb}}
Fix $0 \leq b < \delta/2$ so that Proposition \ref{uniformZhang} is satisfied for all $t_1 \not= t_2$ in $\Qbar\setminus\{0,1\}$.  Now fix $t_1 \not= t_2$ in $\Qbar\setminus\{0,1\}$ and a number field $K$ containing $t_1$ and $t_2$.   Set
 $$F = \{x \in \P^1(\Qbar):  \hat{h}_{t_1}(x) + \hat{h}_{t_2}(x) \leq b\} \setminus \{\infty\} = S(b,t_1,t_2) \setminus\{\infty\},$$ 
so $F$ is a finite, $\Gal(\Kbar/K)$-invariant set with 
	$$\hat{h}_{t_i}(F) \leq b$$
for $i = 1, 2$.  At each non-archimedean place $v$ of $K$, we set
	$$\eta_v := \min \{ 1, |t_1(t_1-1)|_v, |t_2(t_2-1)|_v \}.$$
Now fix $\eps'>0$.  For each archimedean $v$, we set
   $$\eta_v:=c(\eps')\min_{i=1,2}\min\{|t_i|_v^2, |t_i-1|_v^2, |t_i|_v^{-2}\}$$
where the constant $c(\eps')$ is from Proposition \ref{arch regularization}.  Let $\eta = \{\eta_v\}_v$; observe that $\eta_v = 1$ for all but finitely many $v$, and 
\begin{equation} \label{eta bound}
\sum_{v \in M_K} - r_v \log \eta_v \leq 2\, \left(h(t_1)+h(1-t_1)+h(t_2)+h(1-t_2)\right)-\frac12 \log c(\eps').
\end{equation}

For non-archimedean $v$, the explicit form of the measure $\mu_{t_i, v}$ (described in Section \ref{nonarchsection}) implies that
	$$(\mu_{t_i,v}, m_{F, \eta,v})_v =  (\mu_{t_i,v}, [F]_v)_v$$
for this choice of $\eta$, because the potentials for $\mu_{t_i,v}$ are constant on disks of radius $\eta_v$.  

We thus obtain from Proposition \ref{upper bound with F} that 
\begin{eqnarray*} \hat h_{t_i} \cdot h_{F,\eta} &\leq&  b \; + \sum_{v \in M^0_K} r_v \frac{-\log \eta_v}{2|F|}  \\ 
&& \;+\; \sum_{v \in M_K^{\infty}} r_v\, \left(-(\mu_{t_i,v},m_{F,\eta,v})_v + (\mu_{t_i,v}, [F]_v)_v + \frac{-\log \eta_v}{2|F|}\right)
\end{eqnarray*} 
for $i = 1,2$, where $M_K^0$ denotes the non-archimedean places and $M_K^\infty$ the archimedean places.

We have for $v \in M_K^{\infty}$ that
$$-(\mu_{t_i,v}, m_{F,\eta,v})_v +  (\mu_{t_i}, [F]_v)_v \leq \eps'  \log\max\{|t_i|_v, |t_i|_v^{-1}, |t_i-1|_v^{-1}\} $$
for $i = 1,2$ by Proposition \ref{arch regularization}.

Since the logarithmic Weil height satisfies $2\, h(x) = \sum_v |\log|x|_v|$, we thus obtain 
\begin{eqnarray*}
 \hat h_{t_i} \cdot h_{F,\eta} &\leq& b +  2\eps'\left(h(t_i)+h(t_i-1)\right) \\
  &&  + \; \frac{2\, \left(h(t_1)+h(1-t_1)+h(t_2)+h(1-t_2)\right)-\frac12 \log c(\eps')}{|F|}
\end{eqnarray*}
for $i = 1, 2$.  Since $h(1-t_i) \leq h(t_i) + \log 2 \leq h(t_1, t_2) + \log 2$ for $i = 1,2$, this inequality becomes 
   $$\hat h_{t_i} \cdot h_{F,\eta} \; \leq\; b+ 2\eps'\left(2 h(t_1, t_2)+\log 2 \right) +\frac{1}{|F|} (8 h(t_1, t_2)+4 \log 2 -\frac12 \log c(\eps'))$$
for $i = 1,2$.

By the triangle inequality \eqref{triangle}, we have 
\begin{eqnarray*} \left(\hat h_{t_1}\cdot \hat h_{t_2}\right)^{1/2}&\leq& \left(\hat h_{t_1}\cdot h_{F,\eta}\right)^{1/2}+\left(\hat h_{t_2}\cdot h_{F,\eta}\right)^{1/2}\\
&\leq& 2 \left(  b+ 2\eps'\left(2 h(t_1, t_2)+\log 2 \right) +\frac{1}{|F|} (8 h(t_1, t_2)+4 \log 2 -\frac12 \log c(\eps'))\right)^{1/2}, 
\end{eqnarray*} 
so 
\begin{eqnarray} \label{pairing eta height}
\hat h_{t_1}\cdot \hat h_{t_2} &\leq & 4 \left( b+ 2\eps'\left(2 h(t_1, t_2)+\log 2 \right) +\frac{1}{|F|} (8 h(t_1, t_2)+4 \log 2 -\frac12 \log c(\eps'))\right) \nonumber \\
&=&  4b + \left( \frac{32}{|F|} + 16 \eps' \right) h(t_1, t_2) + \frac{16 \log 2 - 2 \log c(\eps')}{|F|} + 8 \eps' \log 2.
\end{eqnarray}

Fix any $\eps>0$, and choose $\eps' < \eps/16$.  Since $|F|=|S(b,t_1,t_2)|-1\geq 2$, we can find a large constant $C(\eps)$ satisfying 
	$$\frac{32}{|F|} + 16 \eps'  \leq  \eps+\frac{C(\eps)}{|S(b,t_1,t_2)|}$$
and
  $$\frac{16 \log 2 - 2 \log c(\eps')}{|F|} + 8 \eps' \log 2 \; \leq\; \eps+\frac{C(\eps)}{|S(b,t_1,t_2)|}.$$
The inequality \eqref{pairing eta height} then yields 
  $$ \hat h_{t_1}\cdot \hat h_{t_2} \leq  4b + \left( \eps + \frac{C(\eps)}{|S(b,t_1, t_2)|} \right) (h(t_1, t_2) + 1),$$
concluding the proof of Theorem \ref{boundingNb}.

%%%%%%
\bigskip
\section{Proof of Theorem \ref{Legendretheorem}} \label{Legendresection}

In this section, we deduce Theorem \ref{Legendretheorem} from Theorems \ref{uniformlowerbound}, \ref{pairingasymptotics}, \ref{boundingN} for algebraic values of $t_1$ and $t_2$; we then extend the result to hold for parameters $t_i$ in $\C$, via a specialization argument.  In fact, we prove the following stronger result over $\Qbar$:  

\begin{theorem} \label{LegendreBogomolov}
There exist constants $B$ and $b>0$ so that 
	$$\big| \{x \in \P^1(\Qbar):  \hat{h}_{t_1}(x) + \hat{h}_{t_2}(x) \leq b\} \big| \; \leq \; B$$
for all $t_1 \not= t_2$ in $\Qbar\setminus\{0,1\}$.  
\end{theorem}

\subsection{Proof of Theorem \ref{LegendreBogomolov}} 

Let $\delta > 0$ be as in Theorem \ref{uniformlowerbound} so that
$$\hat{h}_{t_1} \cdot \hat{h}_{t_2} \geq \delta$$
for all $t_1 \not= t_2$ in $\Qbar\setminus\{0,1\}$.  Fix 
	$$0 < b < \delta/8$$ 
so that, from Proposition \ref{uniformZhang}, the set 
  $$S(b,t_1, t_2) = \{x \in \P^1(\Qbar):  \hat{h}_{t_1}(x) + \hat{h}_{t_2}(x) \leq b\}$$
is finite for all $t_1 \not= t_2$ in $\Qbar\setminus\{0,1\}$.  
Let $h(t_1, t_2)$ be the naive logarithmic height on $\mathbb{A}^2(\Qbar)$.  Fix $H> \frac{2 \beta}{\alpha}$ for the $\alpha, \beta$ of Theorem \ref{pairingasymptotics} and such that 
\begin{equation} \label{big H} \frac{H - 8b/\alpha}{H+1} > 3/4.
\end{equation}

Suppose that $t_1 \ne t_2 \in \overline{\mathbb{Q}}$ satisfy $h(t_1, t_2) \geq H$.  Then for $\eps = \frac{\alpha}{4}$, there exists by Theorem \ref{boundingNb} a constant $C$ such that 
$$\hat{h}_{t_1} \cdot \hat{h}_{t_2} \leq 4b + \left( \frac{\alpha}{4} + \frac{C}{|S(b,t_1, t_2)|} \right) (h(t_1, t_2) + 1).$$
On the other hand, by Theorem \ref{pairingasymptotics} and the choice of $H$, we have
$$\frac{\alpha}{2} \, h(t_1, t_2) \leq \alpha h(t_1, t_2) - \beta \leq \hat{h}_{t_1} \cdot \hat{h}_{t_2}.$$
Therefore 
$$\frac{\alpha}{2} \, h(t_1, t_2) \leq 4b + \left( \frac{\alpha}{4} + \frac{C}{|S(b,t_1, t_2)|} \right) (h(t_1, t_2) + 1),$$
and so
$$|S(b, t_1, t_2)| \leq \frac{C}{ \left( \frac{\alpha h/2  - 4b}{h+1} \right) - \frac{\alpha}{4}} = \frac{C}{  \frac{\alpha}{2} \left(\frac{ h  - 8b/\alpha}{h+1} \right) - \frac{\alpha}{4}} \leq \frac{8C}{\alpha}$$
for $h := h(t_1,t_2) \geq H$, from \eqref{big H}.

Suppose now that $t_1 \ne t_2 \in \overline{\mathbb{Q}}$ satisfy $h(t_1, t_2) < H$.  Set $\eps' = \frac{\delta}{4(H+1)}$, and find a constant $C'$ as in Theorem \ref{boundingNb} so that
$$\hat{h}_{t_1} \cdot \hat{h}_{t_2} \leq  4b + \left( \eps' + \frac{C'}{|S(b,t_1, t_2)|} \right) (h(t_1, t_2) + 1),$$
and thus, since $b < \delta/8$, we have
$$\delta/2 < \delta - 4b \; \leq\;  \left( \eps' + \frac{C'}{|S(b,t_1, t_2)|} \right) (h(t_1, t_2) + 1).$$
We conclude that 
$$|S(b,t_1, t_2)| \leq \frac{4(H+1)C'}{\delta},$$
providing a uniform bound also for $t_1$ and $t_2$ satisfying $h(t_1, t_2) < H$.  This completes the proof of Theorem \ref{LegendreBogomolov}.

\subsection{Specialization:  proof of Theorem \ref{Legendretheorem}}
We implement a standard specialization argument to deduce Theorem \ref{Legendretheorem} from Theorem \ref{LegendreBogomolov}. Note that the division polynomials for the Legendre curve $E_t$ have coefficients in $\mathbb{Q}[t]$; see, for example, \cite[Exercise 3.7]{Silverman:elliptic}.  Let $B$ be the uniform bound obtained in Theorem \ref{LegendreBogomolov}, so that
\begin{equation}\label{common torsion number} 
|\pi(E_{t_1}^\mathit{tors}) \cap \pi(E_{t_2}^\mathit{tors})| = \big| \{x \in \P^1(\Qbar):  \hat{h}_{t_1}(x) = \hat{h}_{t_2}(x) = 0\} \big|  \leq B
\end{equation}
for all $t_1\neq t_2\in \Qbar\setminus \{0,1\}$.  Assume that there exist $t_1\neq t_2\in \C\setminus \{0,1\}$ with 
   $$N(t_1, t_2) := |\pi(E_{t_1}^\mathit{tors}) \cap \pi(E_{t_2}^\mathit{tors})| > B$$
and $t_1$ transcendental.  If $x \in \pi(E_{t_1}^\mathit{tors}) \cap \P^1(\Qbar),$ then $x \in \pi(E_{t}^{\mathit{tors}})$ for all $t\in \C\setminus \{0, 1\}$ as it is a root of a division polynomial.  It follows that there is at least one non-algebraic point $x \in \pi(E_{t_1}^\mathit{tors}) \cap \pi(E_{t_2}^\mathit{tors})$, as only $x = 0, 1, \infty$ are torsion images for all $t \in \C\setminus\{0,1\}$ \cite[Proposition 1.4]{DWY:Lattes}.

Now let
	  $$S:=\{x_1, x_2, \cdots, x_N\} = \pi(E_{t_1}^\mathit{tors}) \cap \pi(E_{t_2}^\mathit{tors}),$$
where $N = N(t_1, t_2)$, and assume that $x_1$ is transcendental.  Because it is a torsion image for both parameters, $\Q(x_1, t_1, t_2)$ and therefore also the field 
	$$L := \Q(t_1, t_2, x_1,\cdots, x_N)$$ 
are of transcendence degree one.  Consequently $L$ is isomorphic to a function field $k=K(X)$ for a number field $K$ and an algebraic curve $X$ defined over $\Qbar$.  Via the identification of $L$ with $k$, there exists an algebraic point $\gamma \in X(\Kbar)$ with distinct specializations $x_i(\gamma) \in \P^1(\Qbar)$ for $i = 1, \ldots, N$ and 
  $$t_1(\gamma)\neq t_2(\gamma)\in \Qbar\setminus \{0, 1\}.$$
The division relations in $L$ imply that the specializations $E_{t_1(\gamma)}$ and $E_{t_2(\gamma)}$ have at least $N$
common torsion images, contradicting \eqref{common torsion number}.  Therefore, we must have 
  $$|\pi(E_{t_1}^{\mathit{tors}}) \cap \pi(E_{t_2}^{\mathit{tors}})|\leq B$$
for all $t_1\neq t_2\in \C\setminus \{0,1\}$, and the proof of Theorem \ref{Legendretheorem} is complete.

\subsection{Common torsion images}

We obtain the following immediate corollary of Theorem \ref{Legendretheorem}, which is a special case of Conjecture \ref{BFTconj}.  Recall that a standard projection from elliptic curve $E$ to $\P^1$ is any degree-two branched cover that identifies each point $P\in E$ with its inverse $-P$.  

\begin{cor} \label{BFT special case}
There exists a uniform bound $B$ such that 
$$|  \pi_1(E_1^{\mathit{tors}}) \cap \pi_2(E_2^{\mathit{tors}}) | \leq B$$
for any pair of elliptic curves $E_i$ over $\C$ and any pair of standard projections $\pi_i$ for which 
	$$|\pi_1(E_1[2]) \cap \pi_2(E_2[2])| = 3.$$
\end{cor}

\begin{proof}
By fixing coordinates on $\P^1$, we may assume that $\pi_1(E_1[2]) \cap \pi_2(E_2[2]) = \{0,1, \infty\}$.  For each $e\in E_i[2]$ the composition $\pi_i^e(P) = \pi_i(P+e)$ is again a standard projection and satisfies $\pi_i^e(E_i^\mathit{tors}) = \pi_i(E_i^\mathit{tors})$.  Therefore, we may assume that $\pi_i(O_i) = \infty$ for the origin $O_i$ of $E_i$, $i = 1,2$.  Putting each $E_i$ into Legendre form now shows that the corollary follows from Theorem \ref{Legendretheorem}.  
\end{proof}  

%%%%%%
\bigskip
\section{Proof of Theorems \ref{MMtheorem} and \ref{uniformBogomolov}} \label{MMtheoremsection}

Throughout this section, we let $\mathcal{L}_2$ denote the hypersurface in the moduli space $\cM_2$ consisting of all genus 2 curves $X$ over $\C$ that admit a degree-two map to an elliptic curve; see, e.g., \cite{Shaska:Volklein} for details on $\mathcal{L}_2$.  The surface $\mathcal{L}_2$ consists of all $X$ whose Jacobians admit real multiplication by the real quadratic order of discriminant 4, as explained in the proof of \cite[Theorem 4.10]{McMullen:genus2}.  

For any smooth, compact, genus 2 curve over $\C$, and for any Weierstrass point $P$ on $X$,
	$$|j_P(X) \cap J(X)^{\mathit{tors}}| \geq 6$$
as the difference of two Weierstrass points is torsion.  On the other hand, any curve $X$ of genus $g \geq 2$ has $|j_P(X) \cap J(X)^{\mathit{tors}}| \leq 2$ for all but finitely many $P$, by Baker and Poonen \cite{Baker:Poonen:torsion}, so an Abel-Jacobi map based at a Weierstrass point has in this sense a large number of torsion images.  

In this section we deduce Theorem \ref{MMtheorem} from Corollary \ref{BFT special case}, providing a uniform upper bound on $|j_P(X) \cap J(X)^{\mathit{tors}}|$ for all $X$ in $\mathcal{L}_2$.  We also deduce Theorem \ref{uniformBogomolov} from Theorem \ref{LegendreBogomolov}.

\subsection{Genus 2 curve from a pair of elliptic curves} \label{from 1 to 2}
Suppose that $\pi_1: E_1 \to \P^1$ and $\pi_2: E_2 \to \P^1$ are standard projections on elliptic curves $E_i$ such that 
	$$|\pi_1(E_1[2]) \cap \pi_2(E_2[2])| = 3,$$
as in Corollary \ref{BFT special case}.  Recall that standard projections are degree-two branched covers $\pi: E \to \P^1$ such that $\pi(P) = \pi(-P)$ for all points $P \in E$, and so they have simple critical points at the four points of $E[2]$.  Consider the diagonal $D \subset \P^1\times\P^1$, and lift $D$ to a curve $C \subset E_1\times E_2$ via $\Pi = \pi_1 \times \pi_2$.   Let $\nu: X \to C$ normalize $C$, noting that the degree four map $\Pi \circ \nu: X \rightarrow D$ has branch locus $\pi_1(E_1[2]) \cup \pi_2(E_2[2])$, with each branch point the image of two points in $X$, each of multiplicity two.   By Riemann-Hurwitz, the genus of $X$ is 2, and by construction, the curve $X$ is in $\mathcal{L}_2$ in $\cM_2$.  Note that $X$ maps to both of the elliptic curves $E_1$ and $E_2$ with degree 2.

\subsection{A pair of elliptic curves from a genus 2 curve}
Here we observe that every $X \in \mathcal{L}_2$ arises from the construction described in \S\ref{from 1 to 2}.  In particular, admitting a degree-two branched cover $X \to E_1$ to an elliptic curve $E_1$ implies that $X$ also admits a second degree-two branched cover $X \to E_2$.  The proof of the following proposition shows how the curve $E_2$ arises:

\begin{prop} \label{surface} Every $X \in \mathcal{L}_2$ is the lift of the diagonal under a product of standard projections $\pi_i$ on elliptic curves $E_i$ for which 
	$$|\pi_1(E_1[2]) \cap \pi_2(E_2[2])| = 3.$$
Moreover, there is a Weierstrass point $Q \in X(\mathbb{C})$ and a degree-four isogeny $\Phi: J(X) \rightarrow E_1 \times E_2$ such that
	$$\Phi \circ j_Q(X) =(\pi_1 \times \pi_2)^{-1} D \;\mbox{ in }\; E_1\times E_2$$
where $D$ is the diagonal in $\P^1\times\P^1$,  $J(X)$ is the Jacobian of $X$, and $j_Q$ is the Abel-Jacobi embedding associated to $Q$.
\end{prop}

\begin{proof}
As noted by \cite{Shaska:Volklein} and attributed to Jacobi \cite{Jacobi:review}, each curve $X \in \mathcal{L}_2$ has an affine model 
$$C: y^2 = x^6 - s_1 x^4 + s_2 x^2 - 1,$$
where the polynomial on the right has non-zero discriminant.  $C$ admits degree two maps $(x, y) \mapsto (x^2, y)$ and $(x, y) \mapsto (1/x^2, iy/x^3)$ to elliptic curves with affine presentation
$$E_1: y^2 = x^3 - s_1 x^2 + s_2 x - 1$$
and 
$$E_2: y^2 = x^3 - s_2 x^2 + s_1 x - 1,$$
respectively, defining a map $\nu: X \to E_1\times E_2$.  For each of these curves, the $x$-coordinate projection $\pi_x$ is standard, so $\pi_1:=\pi_x$ and $\pi_2 := 1/\pi_x$ are standard projections for $E_1$ and $E_2$ respectively.  The projection $\pi_1$ ramifies over $\{ \infty, r_1, r_2, r_3 \}$ and $\pi_2$ ramifies over $\{ 0, r_1, r_2, r_3 \}$, where $\{ r_1, r_2, r_3 \}$ are the distinct, nonzero roots of $x^3 - s_1 x^2 + s_2 x - 1$.  Thus
	$$|\pi_1(E_1[2]) \cap \pi_2(E_2[2])| = 3.$$
Define $\Pi := \pi_1 \times \pi_2$, noting that for $(x, y) \in C$, we have 
	$$\Pi \circ \nu (x,y) = \Pi(x^2, 1/x^2) = (x^2, x^2).$$ 
Thus $\Pi \circ \nu (X) = D$, where $D \subset \mathbb{P}^1 \times \mathbb{P}^1$ is the diagonal.  

Fix $r \in \pi_1(E_1[2]) \cap \pi_2(E_2[2])$, and equip each $E_i$ with a group structure such that the identity lies above $r$.  Observe that the $[-1]$-involution on $E_1\times E_2$ induces the hyperelliptic involution on $X$. In particular, the Weierstrass points on $X$ are the six preimages of $\pi_1(E_1[2]) \cap \pi_2(E_2[2])$ under $\Pi \circ \nu$. Choose $Q \in X$ such that $\Pi(\nu(Q)) = (r, r)$, so that $Q$ is Weierstrass and $\nu$ factors as $\Phi \circ j_Q$ for some isogeny $\Phi: J(X) \to E_1 \times E_2$.  The nontrivial elements of the kernel of $\Phi$ are precisely the three $2$-torsion points in $J(X)$ which are differences of Weierstrass points mapping to the same point in the diagonal $D \subset \P^1\times\P^1$.  Thus $\Phi$ is degree four as claimed, completing the proof.
\end{proof}

\subsection{Proof of Theorem \ref{MMtheorem}}
Fix $X \in \mathcal{L}_2$.  From Proposition \ref{surface}, we have elliptic curves $E_1$ and $E_2$ and a Weierstrass point $Q \in X$ such that 
$$|j_Q(X) \cap J(X)^\mathit{tors}| \leq 16 \, |\pi_1(E_1^\mathit{tors}) \cap \pi_2(E_2^\mathit{tors})|,$$
for a pair of standard projections $\pi_i: E_i\to \P^1$ satisfying $|\pi_1(E_1[2]) \cap \pi_2(E_2[2])| = 3$.  Given any other Weierstrass point $P \in X$, we have $[P - Q] \in J(X)^\mathit{tors}$, so we conclude that 
$$|j_P(X) \cap J(X)^\mathit{tors}| = |j_Q(X) \cap J(X)^\mathit{tors}| \leq 16 B,$$
where $B$ is the constant of Corollary \ref{BFT special case}.

\subsection{Proof of Theorem \ref{uniformBogomolov}}
Fix $X \in \mathcal{L}_2 \subset \cM_2$, defined over $\Qbar$.  From Proposition \ref{surface} there is a pair $t_1 \not= t_2$ in $\Qbar\setminus\{0,1\}$ and an isogeny $\Phi: J(X) \to E_{t_1} \times E_{t_2}$ of degree 4 so that $\Pi\circ \Phi \circ j_Q (X)$ is the diagonal in $\P^1\times\P^1$, where $\Pi = \pi \times\pi$ and $Q$ is a Weierstrass point on $X$.  Recall from \S\ref{canonical height} that the N\'eron-Tate canonical height on $\hat{h}_{E_t}$ on $E_t$ satisfies
	$$\hat{h}_{E_t}(P) = \frac12 \, \hat{h}_t(\pi(P))$$
for all $P \in E_t(\Qbar)$ and each $t\in \Qbar\setminus\{0,1\}$.  

Let 
	$$D = \{O_1\}\times E_{t_2} + E_{t_1} \times \{O_2\}$$
be a divisor on $E_{t_1}\times E_{t_2}$ where $O_i$ denotes the identity element of $E_{t_i}$, and let $L_D$ be the associated line bundle.  Let $L_X = \Phi^* L_D$ on $J(X)$, and let $\hat{h}_{L_X}$ be the associated N\'eron-Tate canonical height on $J(X)(\Qbar)$.  By the functoriality of canonical heights \cite[Theorem B.5.6]{Hindry:Silverman}, we have 
\begin{eqnarray*} 
\hat{h}_{L_X}(x) &=& \hat{h}_{L_D}(\Phi(x)) \\
	&=& \hat{h}_{E_{t_1}}(\Phi(x)_1) + \hat{h}_{E_{t_2}}(\Phi(x)_2) \\
	&=& \frac12 \left(\hat{h}_{t_1}(\pi(\Phi(x)_1) )+ \hat{h}_{t_2}(\pi(\Phi(x)_2)) \right),
\end{eqnarray*}
where $\Phi(x) = (\Phi(x)_1, \Phi(x)_2)$ in $E_{t_1}\times E_{t_2}$.  Restricting to the points $x \in j_P(X)(\Qbar)$, so that $\pi(\Phi(x)_1) = \pi(\Phi(x)_2)$ in $\P^1$, the theorem now follows from Theorem \ref{LegendreBogomolov}.

%%%%%%%%%%%%%%%%%%%%%%%%%%%%%%%%%%%%%%%%%%%%%%%%%%%%%%%%%%%%%%%%%%%%%%%%%%%%%%%%

\bigskip \bigskip
%\bibliographystyle{../tex/bib/math}
%\bibliography{../tex/bib/math}

\def\cprime{$'$}

\end{document}